\documentclass[a4paper,11pt]{article}
\usepackage{amsfonts}
\usepackage{amsmath}
\usepackage{graphics}
\usepackage{graphicx,psfrag}
\usepackage{mathrsfs}
\usepackage{amssymb}
\usepackage{enumerate}
\usepackage{vmargin}
\usepackage{verbatim}
\usepackage{amsfonts}
\usepackage{appendix}
\title{}
\author{}

\newcommand{\eps}{\varepsilon}

\newcommand{\N}{\mathbb N}
\newcommand{\Pd}{\widetilde{\mathbb P}}
\newcommand{\Pc}{\mathbb P}
\newcommand{\Prob}{\mathbb P}
\newcommand{\Gc}{G}
\newcommand{\I}{{\mathchoice {\rm 1\mskip-4mu l} {\rm 1\mskip-4mu l}
{\rm 1\mskip-4.5mu l} {\rm 1\mskip-5mu l}}}
\newcommand{\Fc}{F}
\newcommand{\Xc}{X}
\newcommand{\Yc}{Y}
\newcommand{\Tc}{T}

\newcommand{\Xd}{\widetilde{X}}
\newcommand{\Yd}{\widetilde{Y}}
\newcommand{\Td}{\widetilde{T}}

\newcommand{\Gp}{G^{ev}}
\newcommand{\G}{G}

\newcommand{\cE}{\mathcal E}

\newcommand{\Zd}{{{\mathbb Z}^d}}
\newcommand{\PP}{\mathbb P}
\newcommand{\K}{\mathcal K}

\newtheorem{theorem}{Theorem}[section]
\newtheorem{lem}[theorem]{Lemma}

\newtheorem{rem}[theorem]{Remark}
\newtheorem{pro}[theorem]{Proposition}
\newtheorem{defn}[theorem]{Definition}

\newenvironment{proof}{\smallskip\noindent\emph{Proof.}\hspace{1pt}}%
{\hspace{-5pt}{\nobreak\quad\nobreak\hfill\nobreak$\square$\vspace{8pt}%
\par}\smallskip\goodbreak}

\newcommand{\Z}{\ensuremath{\mathbb{Z}}}

\numberwithin{equation}{section}
\begin{document}
\begin{center}
\textbf{\Large The small world effect on the coalescing time\\ of random walks\\}
\vspace{0.5 cm} 
\begin{tabular}{cccc}
by &Daniela Bertacchi$^1$ &and& Davide Borrello$^{1,2}$\\
&daniela.bertacchi@unimib.it && d.borrello@campus.unimib.it
\end{tabular}
\footnotetext[1]{
\begin{tabular}{l}Dipartimento di Matematica e Applicazioni\\
Universit\`a degli Studi di Milano Bicocca\\
Via Cozzi 53\\ 
20125 Milano, Italy.
\end{tabular}
$^2$\begin{tabular}{l}
Laboratoire de Math\'ematiques Rapha\"el Salem\\
UMR 6085 CNRS-Universit\'e de Rouen\\
Avenue de l'Universit\'e BP.12\\
F76801 Saint-\'Etienne-du-Rouvray, France.
\end{tabular}}
\end{center}
\begin{abstract}
A small world is  obtained from the $d$-dimensional torus
of size $2L$ adding randomly chosen connections between sites, in
a way such that each site has exactly one random neighbour in addition to its
deterministic neighbours. We study the asymptotic behaviour of the meeting time $T_L$
of two random walks moving on this small world and compare it with the result
on the torus. On the torus, in order to have convergence, we have
to rescale $T_L$ by a factor $C_1L^2$ if $d=1$, by $C_2L^2\log L$ if $d=2$
and $C_dL^d$ if $d\ge3$. We prove that on the small world the rescaling
factor is $C^\prime_dL^d$ and identify the constant $C^\prime_d$, proving that
the walks always meet faster on the small world than on the torus if $d\le2$, while
if $d\ge3$ this depends on the probability of moving along the random connection.
As an application, we obtain results on the hitting time to the origin of a single walk
and on the convergence of coalescing random walk systems on the small world.
\end{abstract}


\section{Introduction}
\label{intr}

Graphs provide a mathematical model in many scientific
areas, from physics (magnetization properties
of metals, evolution of gases) to biology (neural networks, disease
spreading) and sociology (social networks, opinion spreading). 
Individuals (atoms, molecules, neurons, animals) are identified with
vertices and an edge drawn between two vertices identifies a relation as
proximity or existence of some sort of contact.
When a part or the totality of the edges are subject to some randomness,
it is natural to deal with random graphs (see \cite{Bollobas_book} for a survey).
One can construct a random graph starting from a deterministic graph either 
by adding random connections, or by removing some connections randomly, 
as in percolation. A particular class of random graphs of the first type are 
\textit{small world graphs}, constructed starting from a 
$d$-dimensional (discrete) torus, whose edges are called \textit{short range connections}, 
adding some random connections, called \textit{long range connections}.

Bollobas and Chung \cite{cf:BollobasChung} first noted that adding 
a random matching in a cycle (i.e. $d=1$), the average distance  
between sites is considerably smaller than on the deterministic graph.
Watts and Strogatz \cite{cf:WattsStrogatz} introduced, as a model for biological 
applications, the random graph obtained in $d=1$ with each site connected
to the ones at Euclidean distance smaller than $m$ and long range connections constructed by 
taking the deterministic ones and by moving with probability $p$ 
one of the end sites to a new one chosen at random.
Another possible construction was introduced by Newmann and Watts 
\cite{cf:NewmannWatts}: they took the same deterministic 
short range connections of Watts and Strogatz, but they added a density $p$ of 
long range connections between randomly chosen sites.
Average distance between sites and clustering coefficient of small world graphs
have been well investigated (\cite{cf:AlbertBarabasi},\cite{cf:BarbourReinert}, 
\cite{cf:WattsStrogatz}). See \cite{cf:Durrett} for a historical introduction of 
small world graphs and main results.

Recently some authors have been focusing on processes taking place on random graphs.
Durrett and Jung \cite{cf:DurrettJung} have studied the contact process on the small world. 
Their version of the small world (which is also the one we study in the present
work) is a generalization of the Bollobas-Chung model: they take
the $d$-dimensional torus $\Lambda^{d}(L)=\Z^d \mod 2L$ with short range connections 
between each pair of vertices at Euclidean distance smaller than $m$. The long range
connections are drawn choosing at random a partition of the $(2L)^d$ vertices in pairs 
and connecting each pair of the partition (see Section \ref{BCSW} for more details about 
the construction). 
Note that all sites have exactly one long range neighbour, which may coincide with a short
range one. Nevertheless with large probability most of the sites have a true long range
neighbour, and we choose a random walk which makes the small world
``stochastically  homogeneous'' (see the definition of the transition matrix $P_S$ in 
Section~\ref{BCSW}). 
The main advantage of such a costruction is that we can associate to the random graph a 
non-random translation invariant graph $\mathcal{B}$, called \textit{big world}.
 To get an idea of how the big world looks like, see Figure 1.
One starts with a copy of $\Z^d$ and then attaches to each site an edge joining this site to
another copy of $\Z^d$ and proceeds by (infinite) iteration. This edge represents the long range
connection (thus for instance in dimension 1, if the long range neighbour of 0 is 3, then 
3 is represented in the big world by the site three steps away from 0 in the first copy of $\Z$, 
but also by the site at the endpoint of the ``long range'' edge attached to 0 -- and indeed
by many other sites).
For more details on this deterministic graph and on its 
relationship with the small world
see \cite{cf:DurrettJung} where the big world was first introduced and
 Section \ref{bigworld} .

One expects that if the distance between sites plays an important role 
(as for random walks, coalescing random walk or the contact process),
a process taking place on a small world will behave differently from the same one on the 
torus.
We consider random walks on the small world and, under some assumptions on the starting sites, we
study the asymptotic behaviour of three sequences of random times: the time $W_L$ after which
a single random walk first hits the origin, the time $T_L$ after which
two random walks first meet and the coalescing time $\tau_L$ of a coalescing random walk starting
from a fixed number of particles. Recall that the coalescing random walk on a graph is a 
Markov process in which $n$ particles perform independent random walks  subject to the rule 
that when one particle jumps onto an already occupied site, the two particles coalesce to one. 
The time when we first have only one particle left is called
 \textit{coalescing time}.

It is natural to compare our results with the corresponding results on the torus:
for the simple symmetric continuous time random walk on the $d$-dimensional torus, 
Cox \cite[Theorem 4]{cf:Cox} proved (under some assumptions on the initial position) 
that for $d= 2$ 
$W_L/C_2(2L)^2\log(2L)$, with $C_2=2/\pi$,  and
for $d\ge3$, $W_L/C_d(2L)^d$, with $C_d$ equal to the expected number of visits to the
origin of a discrete time simple symmetric random walk, converge to an exponential of mean 1.
One can also get the same result for the random walk starting from the 
stationary distribution (this was proved in \cite[Theorem 6.1]{cf:Flatto} in discrete time) 
as a corollary. 

Cox and Durrett (see \cite[Theorem 2]{cf:Durrett_steppingI}) proved a result in 
the $2$-dimensional case under more general conditions 
on the starting point and on the transition matrix for a discrete time random walk. 
The case $d=1$ is slightly different: Flatto, Odlyzko and Wales \cite[Theorem 6.1]{cf:Flatto} 
proved that for the discrete time random walk starting 
from the uniform distribution $W_L/L^2$ converges to a certain law.
It is possible to show that these results also hold in continuous time.

Note that, by the symmetry of the walks on the 
torus, it is easy to show that the meeting time $T_{L}$ of two independent 
random walks $X_{t}$ and $Y_{t}$ on the torus, conditioned to $X_{0}=x$ and $Y_{0}=y$, 
coincides with the law of $2W_{L}$ conditioned to the starting point $x-y$. 
Therefore from Theorems \cite[Theorem 4]{cf:Cox}, \cite[Theorem 2]{cf:Durrett_steppingI}) 
and \cite[Theorem 6.1]{cf:Flatto} one easily deduces
 the asymptotic behaviors of the meeting time 
of two particles.

Since a small world is a random graph, studying random walks on it we have
two sources of randomness: the graph and the walk. We denote by 
 $\mathcal{S}^L$ the random variable whose possible values are the
small worlds of size $L$ and by $S$ (or $S^L$ if we need to stress the dependence
on $L$) one of the possible realizations 
of $\mathcal{S}^L$. If $\Delta$ is the transition matrix of
an adapted (i.e. transition from $x$ to $y$ may occur only if they are connected 
by an edge) and translation invariant symmetric random walk on $\Lambda^d(L)$
and $\beta\in(0,1)$, once  a small world $S$ is fixed,
the random walk we consider on it moves according to $\Delta$ with
probability $1-\beta$ and moves along the long range connection
with probability $\beta$.
We denote by $\mathbf P$ the uniform probability on all the small worlds;
given a small world $S$, we denote by $\Prob^{\mu,\nu}_{S}$ 
the joint law of two  independent continuous time 
random walks starting from the probability distributions $\mu$ and $\nu$.
By $\Prob$ (with no pedex) we denote the average over all small worlds
(see Section \ref{BCSW} for the formal definitions).

We look for results for each graph $S$ in a set of large $\mathbf P$-probability 
(``\textit{quenched}" point of view -- note that it is not possible to have
almost sure results) and average results 
 (i.e. with respect to $\Prob$ -- ``\textit{annealed}" point of view).

Durrett \cite[Chapter $6$]{cf:Durrett} 
proved, for a large class of random graphs with $N$ vertices, 
that for each sequence $\{S^{N}\}_{N\ge0}$ of small worlds chosen in sets of large 
$\mathbf P$-probability, if $T_N$ is the meeting time of two particles starting 
from the stationary distribution, then $T_N/CN$ (for some $C>0$)
converges to the exponential distribution.
In particular such a result holds for the small worlds we consider in dimension one
(clearly with $N=2L$).

With different techniques we prove more accurate results in dimension $d$.
We suppose that the two random walks start respectively from the origin 0
and from a site $x_L$ and we prove convergence to an exponential law
of $T_L/C_d(2L)^d$, under the assumption that either $\{x_L\}_L$ is
constant or moves towards infinity at a sufficiently large speed
(bear in mind that on $S$ there are two distances between two sites $x$ and $y$: 
the Euclidean one $|x-y|$ and the - random - graph distance $d_S(x,y)$).
Moreover we identify the constant $C_d$, which is a fundamental tool
to compare our results with the corresponding ones on the torus in $d\ge3$.

Recall that there exists a deterministic graph called big world $\mathcal B$
which can be mapped onto the small world.
Through the inverse of this map (see Section \ref{bigworld} for details) we associate 
to each site $x \in \Lambda^d(L)$ a unique site $+(x)$ in $\mathcal B$,
and to a random walk on the small world we associate a random walk on $\mathcal B$,
whose law is denoted by $\mathbb{P}_{\mathcal{B}}$.  We denote by
$\Gp_{\mathcal{B}}(x)$ the expected time spent at site $0$
by the random
walk at speed 2 starting from $+(x)$ on $\mathcal B$, namely
\begin{equation}\label{eq:GBevcont}
 \Gp_{\mathcal{B}}(x) :=\int_{0}^{\infty}\Prob^{+(x)}_{\mathcal{B}}(X_{2t}=0)dt.
\end{equation}
Note that such constants depend on the probability $\beta$.
We omit this dependence to avoid cumbersome notation. We also write $0$
for $+(0)$ for simplicity's sake. 
\begin{theorem}
Let 
\[f(x,t)=\exp(-t/\Gp_{\mathcal B}(0))\left\{
\frac{1}{\Gp_{\mathcal B}(0)}\delta_0(x)+\left(1-
\frac{\Gp_{\mathcal B}(x)}{\Gp_{\mathcal B}(0)}\right)(1-\delta_0(x))\right\},\]
and let $g(t)=\exp(-t/\Gp_{\mathcal B}(0))$.
\begin{enumerate}
\item Let $x_L\in\Lambda(L)$ for all $L$ such that $x_L=x$ for all $L$ sufficiently large.
Then 
uniformly in $t\geq 0$
\begin{equation}
 \Prob^{x_L,0}\left( \frac{T_L}{(2L)^d}>t\right)\stackrel{L\to\infty}{\to}
f(x,t)
\label{eq:T1}
\end{equation}
\item Let $\alpha_L\geq (\log \log L)^2$, then uniformly in $t\geq 0$
and $x_L$ such that $|x_L|\ge\alpha_L$, 
\begin{align}
 \Prob^{x_L,0}\left(\frac{T_L}{(2L)^d}>t\right)\to 
g(t).
\end{align}
\item Let $x_L\in\Lambda(L)$ for all $L$ such that $x_L=x$ 
for all $L$ sufficiently large.
For all $\eps>0$
\begin{equation}
  \mathbf P\left(S
: \left|\Prob_S^{x_L,0}\left(\frac{T_L}{(2L)^d}>t\right)
- f(x,t)
\right|
<\eps, \forall t\geq 0\right)\stackrel{L\to\infty}{\to}1.
\label{eq:T1q}
\end{equation}
\item Let $\alpha_L\geq(\log \log L)^2$. For all $\eps>0$
\begin{equation}
  \mathbf P\left(S
:\sup_{\{x_{L}:d_{S}(0,x_{L})\geq \alpha_{L}\}} 
\left|\Prob_S^{x_L,0}\left(\frac{T_L}{(2L)^d}>t\right)
-g(t)
\right|
<\eps,  \forall t\geq 0\right)\stackrel{L\to\infty}{\to}1.
\label{eq:qlim}
\end{equation}
\end{enumerate} 
\label{th:main}
\end{theorem}
The main tools in the proof of this result are: the use of
the Laplace transform (much in the footsteps of \cite{cf:Cox});
the fact that with large $\mathbf P$-probability
a large (but not too large) neighbourhood of a fixed
vertex looks like the big world and the fact that for very 
large times the random walker is approximately uniformly distributed
on the graph. We prove Theorem \ref{th:main} for continuous time
random walks, but for discrete time random walks (which we denote by
$\widetilde X_n$) the same arguments
lead to the corresponding result. The only difference is that instead
of $\Gp_{\mathcal{B}}(x)$ one has the expected number of visits at
even times
\begin{equation}\label{eq:GBevdiscr}
 \widetilde{G}^{ev}_{\mathcal{B}}(x)
 :=
\sum_{n=0}^{\infty}\widetilde{\Prob}^{+(x)}_{\mathcal{B}}(\widetilde{X}_{2n}=0).
\end{equation}
Moreover, by a similar argument one proves the result for the hitting time
to the origin (see  Theorem \ref{th:1walk}). 
As a corollary one can get the law of the meeting time of two random walks and the 
law of the hitting time to the origin of a single walker starting from the uniform distribution.
Note that while on a translation invariant graph one immediately 
deduces the results on the meeting time of
two walkers from the results on the hitting time 
(as on the torus), on random graphs like the small world this is not possible.

We are now able to compare the growth speed of $T_L$ on the torus and on the small
world (when the distance between the two starting points goes to
infinity): depending on the dimension $d$, there are a function $f_d(L)$ and
a constant $C_d$ such that $T_L/C_df_d(L)$ converges in law.
The comparison is summarized in Table 1 (where $\Gp_{\mathbb{Z}^{d}}(0)$
is the expected time spent at 0 by the process $\{X_{2t}\}_{t\ge0}$ on $\Z^d$).
\begin{table}[htbp]\caption{$C_d$ and $f_d$ such that $T_L/C_df_d(L)$ converges in law.}
\begin{center}
\begin{tabular}{|c|c|c|}
\hline
&&\\
$d$ & Torus  & Small world \\
    \begin{tabular}{c}                        
                     \hline  \\
                       1\\
\\
                       2\\
\\
                       $\ge$ 3\\
                       \end{tabular}
 &                       \begin{tabular}{l|l}
                       $f_d(L)$ & $C_d$  \\
                       \hline
                       \\
                       $L^{2}$       &  $1/12$ \\
                       \\
                       $L^{2}\log L$ &  $1/\pi$ \\
                       \\
                       $L^{d}$  & $\Gp_{\mathbb{Z}^{d}}(0)$  \\
                       &\\
                       \end{tabular}
&       \begin{tabular}{l|l}
                       $f_d(L)$ & $C_d$   \\
                       \hline
                       \\
                         $L$     &  $\Gp_{\mathcal{B}}(0) $ \\
                       \\
                         $ L^{2}$       &  $ \Gp_{\mathcal{B}}(0)$ \\
                       \\
                        $L^{d}$ &  $\Gp_{\mathcal{B}}(0)$\\
                        &\\
                       \end{tabular}\\ 
                       \hline                                                                         
\end{tabular}
\end{center} 
\end{table}
If $d\leq 2$ the small world effect is clear 
(adding random connection speeds up the meeting time); if $d\geq 3$ 
 we need to compare the two constants $\Gp_{\mathbb{Z}^{d}}(0)$ and
$\Gp_{\mathcal{B}}(0)$. Recall that these quantities depend on 
$\beta$ (the probability with which the random walk moves along the
long range connection). We prove in the Appendix that 
if $\beta$ is small then $\Gp_{\mathbb{Z}^{d}}(0)> \Gp_{\mathcal{B}}(0)$, 
while if $\beta$ is close to $1$ then $\Gp_{\mathbb{Z}^{d}}(0)< \Gp_{\mathcal{B}}(0)$.
Thus two particles meet faster on the small world than on the torus
if $\beta$ is small, but meet faster on the torus if $\beta$ is large.
This means that where the limiting space $\Z^d$ is transient ($d\ge3$), a small
probability of taking a connection towards a distant site (the long range neighbour)
makes it easier to meet, but if this probability is too large then it is
easier for the two walkers to get lost instead of meeting.
Unfortunately identifying the value of $\beta$ at which the inequality 
between $\Gp_{\mathbb{Z}^{d}}(0)$ and $\Gp_{\mathcal{B}}(0)$ reverses seems
a difficult task, which goes beyond the aim of this paper.
One strategy could be try to find numerical approximations of the two constants
by evaluating with a combinatorial procedure 
the $n$-step return probabilities on $\Z^d$ up to time $n_0$ and substituting
in \eqref{eq:GBevcont} this evaluation up to step $n_0$ and 
the asymptotic value of the return probabilities
(see for instance \cite[Theorem 13.10]{cf:Woess}) for $n>n_0$. Of course
this has to be repeated for a large set of values of $\beta$ and one also needs to tackle the question
of how large $n_0$ needs to be in order to make the error small.

The third random time we are interested in is the coalescing time.
In \cite[Chapter $6$]{cf:Durrett}, the author sketched a proof that the number 
of particles of a normalized $n$-coalescing random walk (that is with $n$ particles at time $0$) 
starting from the stationary distribution, moving according to the simple symmetric
random walk, in $d=1$, on the small world, converges to the Kingman's coalescent. 
Recall that the Kingman's coalescent is a Markov process starting from $n$ individuals 
without spatial structure: each couple has an exponential clock with mean $1$ after 
which the two particles coalesce (see \cite{cf:Cox}, \cite{cf:CoxGriffeath} 
and \cite{cf:Tavare}).
We use Theorem \ref{th:main} to get new information about the number of particles 
$(|\xi_{t}(A)|)_{t\geq 0}$ of the coalescing random walk $(\xi_{t}(A))_{t\geq 0}$ 
starting from $A=\{x_{1},\ldots,x_{n}\}$, $x_{i}\in \Lambda^d(L)$ for $1\leq i\leq n$ 
in continuous time, extending the previous result to $d$-dimensional 
small worlds with general transition probabilities and more general initial 
distance between particles. We prove the following, where $M$ is the number 
of deterministic neighbours of each site (depending on the model, $M=2d+1$ or
$M=(2m+1)^d$).
\begin{theorem}
Let $h_{L}\geq (\log \log L)^{2}$ such that $\lim_{L\to\infty}M^{4h_{L}}/(2L)^{d}=0$, 
then for each $A=\{x_{1},\ldots,x_{n}\}\subset \Lambda^d(L)$ with $|x_{i}-x_{j}|\geq h_{L}$ 
for $i\neq j$, $T>0$ there exists a sequence of sets $\{H^L\}_{L}$ of small world graphs 
such that $\mathbf P(H^L)\stackrel{L\to \infty}{\to} 1$ and for each sequence $\{S^L\}_{L}$, 
$S^L\in H^L$, uniformly in $0\leq t\leq T$ 
\begin{equation}
\left|\PP^{A}_{S^L}\
\left(|\xi_{s_{L}t}(A)|< k\right)-P_{n}\left(D_{t}< k\right)\right|\stackrel{L\to \infty}{\to}0, \qquad
k=2,\ldots,n.
\label{eq:coalescing}
\end{equation}
where $\PP^{A}_{S^L}$ is the law of $(\xi_{t}(A))_{t\geq 0}$ on $S^L$, 
$s_{L}=(2L)^{d}\Gp_{\mathcal{B}}(0)$ and $P_n$ is the law of the number of particles $D_{t}$ at time $t\geq 0$ in a Kingman's 
coalescent starting from $n$ particles. 
\label{npart}
\end{theorem}

We remark that the small world we consider is a random graph where each site has
 a single long range connection. One can show analogous results for random graphs with a
fixed number $K>1$ (not depending on $L$) of long range connections per site, added to
 the $d$-dimensional torus or to a translation invariant finite graph. 
The exponential limit will have a different parameter which we guess would be the 
expected time spent at the origin on a different big world structure.

We give here a brief outline of the paper.
In Section \ref{BCSW} we give the formal definitions needed in the 
sequel and give some technical results. In Subsection \ref{BCSW} we formally
define the small world (actually two versions of it, depending on the notion of deterministic
neighbourhood one chooses) and the random walk on it.
 In Subsection \ref{bigworld} we describe the big world and its relationship
with the small world. Moreover we prove that with large probability a ball
of radius $t(L)$ (with $t(L)$ which does not grow too fast) in the small world looks 
exactly like the corresponding
ball in the big world (that is there are no long range connections reaching inside
the ball -- see Proposition \ref{th:xgood}). Proposition \ref{th:stimedS}
gives useful lower bounds on the probability that the graph distance and the
Euclidean distance between two points are equal, and on the probability, if the latter
is large, that also the former is large. As we already mentioned, one of the
keys in our proofs is that when time is relatively small, thanks to  Proposition \ref{th:xgood}
the random walker moves with large $\mathbf P$-probability as if she were on the big world.
On the other hand, for large times we use the fact that she is close to the stationary
distribution. In Subsection \ref{convergence}
we state Proposition \ref{th:conv-estim} which
is an estimate on the speed of convergence to equilibrium. Its proof 
uses known estimates, involving the isoperimetric constant.
 This is the reason why we need the results in Subsection \ref{iso_constant},
which roughly speaking say that with large $\mathbf P$-probability
the isoperimetric constant is large.
Section \ref{hitting_time} is devoted to the estimates of the asymptotic behaviour
of the Laplace transforms  of the meeting time of
two random walks, one starting at $x$ and the other starting at $0$:
$F^L(x,\lambda)$ is the ``annealed'' transform (i.e. with respect to $\Prob$)
and $F^L_S(x,\lambda)$ is the ``quenched'' transform (i.e. with respect to $\Prob_S$).
To obtain these estimates, we need to evaluate the Laplace transforms
of the time spent together of two random walks, namely
$G^L(x,\lambda)$ and $G^L_S(x,\lambda)$. These results are used in
Section \ref{hit_time} where we prove Theorem \ref{th:main} and the result
on the hitting time of the origin.
In Section 
\ref{coalescing} we introduce the coalescing random walk and we prove the convergence 
theorem to Kingman's coalescent. Finally in the Appendix we compare
 $\Gp_{\mathbb{Z}^{d}}(0)$ and $\Gp_{\mathcal{B}}(0)$, 
which allows to compare our results with the ones on the meeting time on the $d$-dimensional 
torus when $d\geq 3$. 

\section{Preliminaries}

\subsection{The small world}
\label{BCSW}

The vertices of the random graph are the ones of the $d$-dimensional torus, which we denote by 
$$\Lambda(L)=\Lambda^d(L)=(\Z \mod 2L)^d,$$
when there is no ambiguity, we will omit the dependence on $d$.

The set of edges $\cE^L$ of the graph is partly deterministic (short range
connections) and partly random (long range connections). Note that
we consider nonoriented edges, that is, if $(x,y)\in\cE^L$ then
also $(y,x)\in\cE^L$ (thus we identify edges with subsets of order two).

We will consider two kinds of short range connections, one between neighbours
(i.e.~vertices $x,y$ such that $\Vert x-y\Vert_1=1$), and the other 
between vertices  $x,y$ such that $\Vert x-y\Vert_\infty\le m$: 
the corresponding neighbourhoods are 
$$
\begin{array}{ll}
\mathcal N(x)=\{y \in \Lambda(L):\Vert x-y\Vert_{1}= 1\},& x \in \Lambda(L),\\
\\
\mathcal N_{m}^\infty(x)=\{y \in \Lambda(L):\Vert x-y\Vert_{\infty}\leq m\},& x \in \Lambda(L),
m\in\N.
\end{array}
$$
For all $x,y\in\Lambda(L)$ we denote by $d_S(x,y)$ the graph distance between $x$ and $y$.
Let $\Omega$ be the set of all partitions of the set of $\Lambda(L)$ 
into $(2L)^{d}/2$ subsets of cardinality two. Let $\mathbf P$ be the uniform
probability on $\Omega$: the random choice of 
$\omega\in \Omega$ represents the choice of the set of long
range connections (some of which may coincide with short range ones). Note
that both $\Omega$ and $\mathbf P$ depend on $L$.

\begin{defn}
 Let $\mathcal G^L$ be the family of all graphs with set of vertices $\Lambda(L)$.
The small world $\mathcal S^L$ is a random variable $\mathcal S^L({\omega}):\Omega\to \mathcal G^L$ such 
that $\mathcal S^{L}(\omega)
=\left(\Lambda(L),\cE^L(\omega)\right)$, where
\[
 \cE^L(\omega)=\omega\cup\{
\{x,y\}:x\in\Lambda(L),y\in\mathcal N(x)\}.
\]
The set of edges of the small world $\mathcal S^L_m({\Omega})$ is defined as
\[
 \cE^L_m(\omega)=\omega\cup\{
\{x,y\}:x\in\Lambda(L),y\in\mathcal N_m^\infty(x)\}.
\]
We denote by $\mathcal S^L({\Omega})=\{\mathcal S^L({\omega}):{\omega} 
\in {\Omega}\}$ and by $\mathcal S_m^L({\Omega})=\{\mathcal S_m^L({\omega}):{\omega} 
\in \widetilde{\Omega}\}$.\\
For any fixed $\omega$, given two short range neighbours $x$ and $y$, 
we write $x\sim^{SR} y$; if they are long range neighbours we write $x\sim^{LR}  y$
(it may happen that $x\sim^{SR}  y$ and $x\sim^{LR}  y$ at the same time).
\end{defn}
Note that $\mathbf P$ clearly defines a probability measure on $\mathcal G^L$: with a slight abuse of notation we denote this measure with 
$\mathbf P$ as well. Given $\omega$, we will also call ``small world'' the graph 
$\mathcal S^L(\omega)$.
For the sake of simplicity we will focus here on the case $\mathcal S^L$,
but our proofs can be adapted to $\mathcal S^L_m$. Moreover,
when there is no ambiguity, we will write $\mathcal S$ and $\mathcal S_m$
instead of $\mathcal S^L$ and $\mathcal S^L_m$.
\begin{rem}
 We note that the small world could be defined imposing that we consider
as probability space $\Theta\subset\Omega$, the family of partitions where no couple is a short range connection (thus
the random graph has fixed degree), instead of $\Omega$.
The results of the paper would not be different.
\end{rem}

We consider discrete and continuous time random walks on small worlds, here is the definition
regarding the discrete ones.

\begin{defn}\label{def:law}
Let $\Delta$ be an adapted, symmetric and translation invariant
transition matrix on the torus, $A_S$ be the (random) matrix where the $x,y$ entry
is 1 if and only if $x$ and $y$ are long range neighbours and 0 otherwise, and 
$\mu$ be a probability measure on $\Lambda(L)$.
\begin{enumerate}
\item
  Given a small world $S$, the transition matrix of the walk is
$P_{S}=\beta \Delta+(1-\beta)A_S$ and we denote by
$\Pd^\mu_{S}$ the law of the discrete time random walk on $S$ with initial
probability $\mu$ and transitions ruled by $P_{S}$.
If $\mu=\delta_{x_0}$ we write $\Pd^{x_0}_{S}$.
\item
We denote by
$\Pd^\mu$ the average of  $\Pd^\mu_{S}$ with respect to $\mathbf P$, that is
\[
 \Pd^\mu(
\mathcal C(x_0,\ldots,x_n))=\sum_{S\in 
\Omega}
\mathbf P(S)\mu(x_0)p_{S}(x_0,x_1)\cdots 
p_{S}(x_{n-1},x_n),
\]
where  
$\mathcal C(x_0,\ldots,x_n)$ is the set of all infinite sequences of vertices
where the first $n$ coordinates coincide with
 $(x_0,\ldots,x_n)$.
\end{enumerate}
\end{defn}
We construct the continuous time version $\Xc_{t}$ of the random walk $\Xd_{t}$ by continuation. 
In other words we define $\Xc_{t}:\stackrel{d}{=}\Xd_{N_{t}}$ where $N_{t}$ 
is a Poisson process with rate $1$ independent of $\Xd_{t}$: the law of $\Xc_{t}$ on 
$S$ starting from a probability measure $\mu$ on $\Lambda(L)$ is given by
\begin{equation}
\Pc^{\mu}_{S}(\Xc_t=y)=\sum_{k=0}^{\infty}\frac{e^{-t}t^k}{k!}\Pd_{S}^{\mu}(\Xd_k=y).
\label{Xcont}
\end{equation}
From now on 
$\Delta$, hence also the family of
transition matrices $\{P_{S}\}_{S\in\mathcal S^L(\Omega)}$,
 is considered fixed.

\subsection{The big world}
\label{bigworld}

The small worlds $\mathcal S^L$ and $\mathcal S^L_m$ 
(which are random graphs)
can be mapped into deterministic graphs, the \textit{big worlds}
$\mathcal B$ and $\mathcal B_m$ respectively, as in 
\cite{cf:DurrettJung}.  
We recall here the construction.
The sites are all vectors $\pm(z_{1},\ldots,z_{n})$, with $n \geq 1$ components, for all $n\in \N$
$z_{j}\in \mathbb{Z}^{d}$ and $z_{j}\neq 0 $ for $j < n$.
The edges in $\mathcal B$
are drawn between $+(z_{1},\ldots,z_{n})$ and $+(z_{1},\ldots,z_{n}+y)$ 
if and only if $y\in\mathcal N(0)$; for $\mathcal B_m$ we consider
$y\in\mathcal N_m^\infty(0)$ (we call these edges short range
connections). The same is done between $-(z_{1},\ldots,z_{n})$ and $-(z_{1},\ldots,z_{n}+y)$.
Moreover $+(z_{1},\ldots,z_{n})$ has a long range neighbour, namely
$$
\begin{array}{ll}
+(z_{1},\ldots,z_{n},0) & \text{if }z_n\neq0,\\
+(z_{1},\ldots,z_{n-1}) & \text{if }z_n=0, n\geq 1,\\
-(0) & \text{if }z_n=0,n=1.
\end{array}
$$
Analogously one defines the long range neighbour of $-(z_{1},\ldots,z_{n})$.
Note that the big world is a vertex transitive graph (i.e. the automorphism group acts transitively).
We denote by $|x|$ the graph distance on the big world from $x$ to $+(0)$ and we also write 0
instead of $+(0)$. See Figure 1 (which is taken from \cite{cf:DurrettJung}) for the
case $d=1$.
\begin{figure}[h]\label{fig:bw1}
   \begin{center}
    \includegraphics[width=10cm]{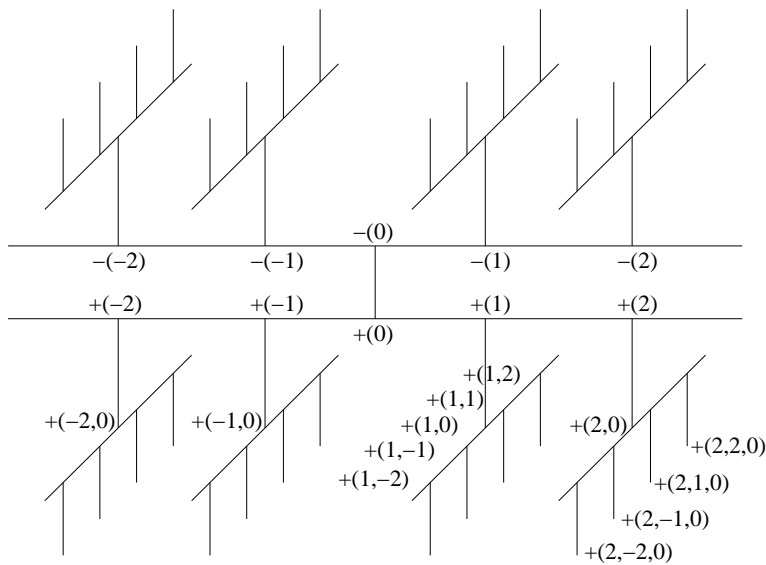}
  \end{center}
 \caption{A portion of the big world in $d=1$.}
\end{figure}

We construct a random map $\phi(\omega)$ which maps the big world onto the
 small world $\mathcal S(\omega)$ in a way such that long range connections
in the big world correspond to long range connections in the small world.
\begin{defn}
Given a small world $S$ and $x\in\Lambda(L)$, let
$LR_{S}(x)$ be the long range neighbour of $x$.
The map $\phi:\Omega\to\Lambda(L)^{\mathcal B}$ is recursively defined as follows:
\[\begin{array}{ll}
 \phi(\omega)(+(z))& =z \mod(2L),\\
\phi(\omega)(-(z))& =LR_{\mathcal S(\omega)}(0)+z \mod(2L),\\
\phi(\omega)(\pm(z_1,\ldots,z_n))& =LR_{\mathcal S(\omega)}\left(\phi(\omega)(\pm(z_1,\ldots,z_{n-1})\right)+z_n \mod(2L).\\
\end{array}
\]
\end{defn}
Note that the transition matrix $P_S$ on the small world naturally 
induces a symmetric and translation invariant discrete time
random walk $\{\Xd_n\}_{n\ge0}$ on the big world (moving with probability $\beta$
along the long range connection and with probability $1-\beta$
according to the transition matrix $\Delta$).
We denote by $\Pd^x$ the law of $\{\Xd_n\}_{n\ge0}$ with initial position $x$.
One can prove, by using Cauchy-Schwarz's inequality, the symmetry and the translational invariance
of the walk, that for all $x\in\mathcal B$ and $n\in\N$, 
\begin{equation}\label{eq:Pbwub}
\Pd^x_{\mathcal{B}}(\Xd_{2n}=0)\le \Pd^0_{\mathcal{B}}(\Xd_{2n}=0);\qquad \Pd^x_{\mathcal{B}}(\Xd_{2n+1}=0)\le \Pd^0_{\mathcal{B}}(\Xd_{2n}=0).
\end{equation}
Using \eqref{Xcont} we get the continuous time version $\{X_t\}_{t\ge0}$ and we have 
that for each $t\geq 0$
\begin{equation}\label{eq:Pbwubcont}
 \begin{array}{ll}
\Pc^x_{\mathcal{B}}(\Xc_{2t}=0)&\le \Pc^0_{\mathcal{B}}(\Xc_{2t}=0).
 \end{array}
\end{equation}

Let $\widetilde{\G}_{\mathcal{B}}(x):=\sum_{n=0}^\infty \Pd_{\mathcal{B}}^{x}(\Xd_{n}=0)$ be the
expected number of visits to 0 of $\{\Xd_{n}\}_{n\ge0}$ (recall that
in \eqref{eq:GBevdiscr} we introduced
${\widetilde G}^{ev}_{\mathcal B}(x)$)
and let ${\G}_{\mathcal{B}}(x)=\int_{0}^{\infty}\Pc_{\mathcal{B}}^{x}(X_{t}=0)dt$ be
the expected time spent at 0 by the continuous time process (recall \eqref{eq:GBevcont}
to compare with  $\Gp_{\mathcal B}(x)$).
We can prove, starting from \eqref{Xcont}, that
$\widetilde{\G}_{\mathcal{B}}(x)=\G_{\mathcal{B}}(x)$ and by a change of variable,
that $\G_{\mathcal{B}}(x)=2\G_{\mathcal{B}}^{ev}(x)$.

Clearly  $\widetilde G^{ev}_{\mathcal{B}}(x)\le \widetilde G_{\mathcal{B}}(x)$ and they coincide if the random
walk has period 2 (in which case they are nonzero only if $|x|$ is even).
Note that if $m=1$ the big world is the Cayley graph
of $\Z^d*\Z_2$ and the random walk on it is transient and $\G_{\mathcal{B}}(x)$ is finite. 
If $m\ge2$ the big world is the Cayley graph
of $\widehat\Z^d*\Z_2$, where $\widehat\Z^d$ has the
$m$-neighbourhood relation, and the random walk is still transient (this can be proven via
the flow criterion, see \cite{cf:Woess}).
Moreover, by \eqref{eq:Pbwubcont}, 
$\Gp_{\mathcal{B}}(x)\le \Gp_{\mathcal{B}}(0)$, and the analogous
inequality holds in discrete time.

We are interested in the event where 
{locally the small world
does not differ from the big world}.
\begin{defn}\label{def:I(x,t)}
If $x\in\Lambda(L)$ and $t>0$, we denote by $I(x,t)$
the event in $\Omega$
\[
I(x,t):= \{\phi_{|B_{\mathcal{B}}(x,t)}\text{ is injective}\},
\]
where $B_{\mathcal{B}}(x,t)$ is the ball of radius $t$ centered at $x$ in the big world.
\end{defn}
Clearly $\mathbf P(I(x,t))$ does not depend on $x$.

\begin{pro}\label{th:xgood}
Let $M=2d+1$ for $\mathcal B$ and $M=(2m+1)^d$
for $\mathcal B_m$ and let $t=t(L)$ be a function of $L$ such that
$M^{4t(L)}=o(L^d)$. 
Then for sufficiently large $L$
  \[
   \mathbf P(I(x,t))\ge1-\frac{CM^{4t}}{L^d}
\stackrel{L\to\infty}{\to}1,
  \]
where $C$ is a positive constant.
\end{pro}
\begin{proof}
Denote by $K_t$ the number of long range connections in $B_{\mathcal{B}}(0,t)$,
and by $J_t$ the total number of vertices in the ball centered at 0
in $\Lambda(L)$ and of radius $t$, which we denote by $B_{\mathcal{B}}(0,t)$. 
Note that the number of vertices in a graph with constant degree 
is always at most the number of vertices in the homogeneous tree
of the same degree. 
Recall that the ball of radius $t$ in the
homogeneous tree of degree $M\ge3$ has exactly $1+M\sum_{j=0}^{\lfloor t\rfloor-1}
(M-1)^j\le 3M^t$ vertices.
Thus we get $K_t\le 3M^t$ and $J_t\le 3M^t$.

Enumerate the long range connections in $B_{\mathcal{B}}(0,t)$ from 1 to $K_t$
and construct the mapping $\phi$. Note that $I(0,t)$
contains the set $A$ of $\omega$ such that the long range connections
in the image of $B_{\mathcal{B}}(0,t)$ in the small world $\mathcal S$ are all sites at
distance at least $2t$ on $\Lambda(L)$.
Thus $\mathbf P(I(0,t))\ge\mathbf P(A)$ and
\[\begin{split}
 \mathbf P(A) & \ge  \frac{(2L)^d-J_{2t}}{(2L)^d}\frac{(2L)^d-2J_{2t}}{(2L)^d}
\cdots \frac{(2L)^d-K_tJ_{2t}}{(2L)^d}\\
& =\prod_{i=1}^{K_t}\left(1-\frac{iJ_{2t}}{(2L)^d}\right)=
\exp\left(\sum_{i=1}^{K_t}\log\left(1-\frac{iJ_{2t}}{(2L)^d}\right)
\right).\end{split}
\]
Note that $\log(1-x)\ge-2x$ if $x\in[0,\bar x]$ for some $\bar x$.
By our choice of $t(L)$, for $L$ sufficiently large we have
that $K_tJ_{2t}/L^d\le \bar x$ and we get, for some positive $C$
and $C^\prime$,
\[
 \mathbf P(A)\ge \exp \left(-\frac{2J_{2t}}{(2L)^d}\sum_{i=1}^{K_t}i
\right)
\ge\exp\left(-C^\prime\frac{J_{2t}K_t^2}{L^d}
\right)\ge \exp\left(-\frac{CM^{4t}}{L^d}
\right)\ge 1-\frac{CM^{4t}}{L^d}.
  \]
\end{proof}

By $d_{\mathcal S}(x,y)$ we denote the (random) graph distance between $x$ and $y$.
Depending on $\omega$, $x$ and $y$, it may happen that
$d_{\mathcal S}(x,y)=d(x,y)$ or $d_{\mathcal S}(x,y)<d(x,y)$. The following
proposition provides probability estimates of these events.

\begin{pro}\label{th:stimedS}
Choose $t$ as in Proposition~\ref{th:xgood}. Then for sufficiently large $L$
 \begin{enumerate}[a.]
 \item if $d(0,x)\le t$, then
\begin{equation}
\mathbf{P}\left(d_{\mathcal S}(0,x)=d(0,x)\right) \ge1-\frac{CM^{4t}}{L^d};
\label{eq:stima>01}
\end{equation}
\item if $d(0,x)>t$, then
\begin{equation}
\mathbf{P}\left(d_{\mathcal S}(0,x)> t\right) \ge 1-\frac{CM^{4t}}{L^{d}}.
\label{eq:stima>02}
\end{equation}
\end{enumerate}
\end{pro}
\begin{proof}
\begin{enumerate}[a.]
\item 
It suffices to note that the event $(d_{\mathcal S}(0,x)=d(0,x))$
contains the event $A$ of the previous proposition.
\item
We note that the event $(d_{\mathcal S}(0,x)>t)$
contains $C_x$ which is the event that
 all the $K_{t/2}$ long range connections in $B_{\mathcal{B}}(0,t/2)$ and
 $B_{\mathcal{B}}(x,t/2)$ 
are mapped by $\phi$ into vertices of $\Lambda(L)$ at distance at least $t$ from
each other and from the balls of radius $t$ centered at 0 and at $x$ in $\Lambda(L)$.
 We work as in Proposition~\ref{th:xgood} to estimate
\[\begin{split}
 \mathbf P(C_x) & \ge  \frac{(2L)^d-2J_{t}}{(2L)^d}\frac{(2L)^d-3J_{t}}{(2L)^d}
\cdots \frac{(2L)^d-K_{t/2}J_{t}}{(2L)^d}.\\
\end{split}
\]
and we proceed in a similar way to get the thesis.
\end{enumerate}
\end{proof}


\subsection{Isoperimetric constant}
\label{iso_constant}

Estimates of the distance between the random walk and the equilibrium
measure involve the isoperimetric constant. Thus we will get bounds
for the \textit{edge isoperimetric constant}
\begin{equation*}
 \iota=\min_{|V|\le n/2}\frac{e(V,V^c)}{|V|},
\end{equation*}
where $n$ is the total number of vertices in the graph and $e(V,V^c)$
is the total number of edges between vertices in
$V$ and $V^c$, where $V$ is a subset of the vertices of the graph.

Given $\alpha>0$, we define
\begin{equation}
Q^L_\alpha:=(S\in\mathcal S^L({\Omega}):\iota(S)>\alpha)
\label{QL}
\end{equation}
We want to prove that there exists $\alpha$ (independent of $L$)
such that $\mathbf{P}(Q^L_\alpha)$ is large when $L$ is large.
 In order to prove this
we need to recall some facts about random graphs.

Take $n$ and $r$ positive integers such that $nr$ is even and consider the
random multigraph with $n$ vertices obtained in the following way:
attach to each vertex $r$ half edges, pick at random
(with uniform probability $\Prob$) a pairing of the $nr$ half edges
and join the half edges which are paired. Note that parallel edges and loops
are possible and that the degree is at most $r$.
We call this multigraph
a random $(n,r)$-configuration. This procedure is usually proposed as
a way to construct, with uniform probability, the random $r$-regular graph
(one has to condition to the event that the multigraph has neither parallel
edges nor loops, i.e. it is a graph), see 
\cite{cf:Bollobas} or \cite{Bollobas_book}.

Let us now recall  \cite[Theorem 6.3.2]{cf:Durrett} (which is inspired by 
\cite[Theorem 1]{cf:Bollobas}): it claims that, given $r$, there exists $\alpha^\prime>0$
such that $\Prob(\text{the }(n,r)\text{-configuration has }\iota\le\alpha^\prime)=o(1)$ as $n$
goes to infinity. 
It is not difficult to modify the proof of Durrett to get that for any fixed
positive integer $l$ one can refine the estimate and obtain 
$o(n^{-l})$.

\begin{pro}\label{th:isorrg}
Let $n$, $r$ and $l$ be positive integers with $nr$ even and let $\Prob$ be the uniform
probability on $(n,r)$-configurations. Then there exists $\alpha^\prime>0$
independent of $n$ and $r$ (one may choose $\alpha^\prime=1/10l$), such that
 $\Prob(\iota\le\alpha^\prime)=o(n^{-l})$. 
\end{pro}

Actually one can prove this proposition for more general random graphs.
Indeed let $h$ be a positive integer in $[1,r-1]$ and
call $(n,r,h)$-configuration the multigraph obtained by a procedure similar
to the one we used for $(n,r)$-configurations. The only difference is that here
$n$ vertices have $r$ half edges each attached, and one vertex has $h$ half edges
attached ($nr+h$ has to be even). It is not difficult to prove that 
Proposition \ref{th:isorrg} holds also for $(n,r,h)$-configurations.

Now we use this result to prove the analog for the small world.
The ideas are taken from  \cite[Theorem 6.3.4]{cf:Durrett}.

\begin{pro}\label{th:isosw}
 Consider the small world $\mathcal S^L$ and fix a positive integer $l$.
Then there exists $\alpha>0$ such that
 $\mathbf P(Q^L_\alpha
)=o(L^{-dl})$. 
\end{pro}

\begin{proof}
We represent the vertices in $\Lambda(L)$ by vertices in $[-L,L)^d\cap\Z^d$
and partition this set into triplets plus possibly a singleton or a couple of
vertices (if $2L \mod 3=1$ or 2 respectively). We enumerate the triplets from
1 to $n=\lfloor (2L)^d/3\rfloor$ and denote them by
 $I_1,\ldots, I_n$.
Note that it is possible to choose the triplets in a way such that each triplet
has a vertex which is a short range neighbour of the other two vertices (see
Figure 2 for the case $d=2$ and $L=4$).
\begin{figure}[htb]\label{fig:triplets}
 \begin{center}
   \includegraphics[height=5cm]{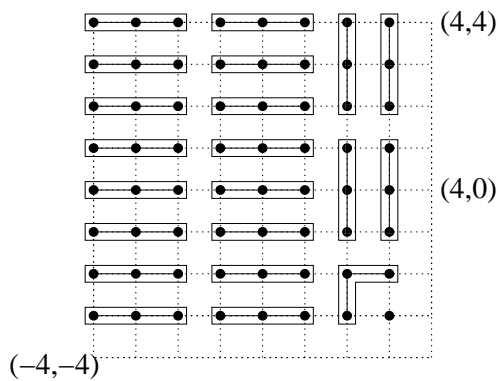}
 \end{center}
\caption{The partition in triplets if $d=2$ and $L=4$.}
\end{figure}

Now choose $A\subset [-L,L)^d\cap\Z^d$
with $|A|\le (2L)^d/2$: we need to  prove that outside
a set of small worlds of $\mathbf P$-probability which is
$o(L^{-dl})$ we have that $e(A,A^c)/|A|>\alpha$. Let 
$J_A=\{j\in\{\,\ldots,n\}: I_j\subset A\}$,
$K_A=\{j\not\in J_A: I_j\cap A\neq\emptyset\}$
and $B_A=\bigcup_{j\in J_A}I_j$.

\textbf{Case 1.} If $|J_A|\le |A|/6$ then $B_A|\le|A|/2$ and $|K_A|\ge |A|/4$.
Thus there are at least $|A|/4$ edges connecting $A$ to $A^c$ ($|K_A|$ is
a lower bound for $e(A,A^c)$).

\textbf{Case 2.} If $|J_A|> |A|/6$ we construct a $(n,3)$-configuration associated to
the small world: there is an edge between $j$ and $k$ for any long range 
edge between $x\in I_j$ and $y\in I_k$.
By Proposition \ref{th:isorrg} outside a set of $\mathbf{P}$-probability $o(L^{-dl})$
we have $e(J_A,J_A^c)\ge\alpha^\prime|J_A|$ for some $\alpha^\prime>0$.

It is enough to show that there is a map $\varphi$ from the set of edges between $J_A$
and $J_A^c$ to the set of edges between $A$ and $A^c$, such that each edge between
$A$ and $A^c$ has at most two preimages.
Let $j\in J_A$, $k\in J_A^c$ and let there be an edge between them. Then there
exists $x\in I_j\subset A$, $y\in I_k$ such that $x\sim^{LR} y$.
If $y\not\in A$ then $\varphi((j,k))=(x,y)$. If $y\in A$ and it has
a short range neighbour $z\in I_k\cap A^c$ then we choose $\varphi((j,k))=(y,z)$ and 
$(y,z)$ has no other preimages.

If $y$ does not have a short range neighbour in $I_k\cap A^c$
we know that it has a short range neighbour $z\in I_k\cap A$ which is a short
range neighbour of $z^\prime\in I_k\cap A^c$. 
 In this case, $\varphi((j,k))=(z,z^\prime)$ and $(z,z^\prime)$ might have at most another
preimage (the edge between some $i$ and $k$ originated by the long range edge
between $z$ and some $x^\prime\in I_i$.
\end{proof}

\subsection{Convergence to equilibrium}
\label{convergence}

Note by symmetry that the reversible distribution of the walk on
$\mathcal S^L$ is the uniform probability $\pi$ on $\mathcal S^L$.
\begin{pro}\label{th:conv-estim}
Let $\{X_t\}_{t\ge0}$ be the continuous time random walk on
the small world 
(recall Definition~\ref{def:law} and equation \eqref{Xcont}).
Fix $l\in\N$ and pick $\alpha$ as in Proposition~\ref{th:isosw}. 
There exists $\gamma>0$ (depending only on $\alpha$, $\Delta$
and $\beta$) such that 
\begin{align}
\max_{x,y}\left|\Pc_S^x(\Xc_t=y)-\pi(y)\right|&\le e^{-\gamma t},\qquad \text{for all }S\in Q^L_\alpha;
\label{eq:estim_cont1}\\
\max_{x,y}\left|\Pc^x(\Xc_t=y)-\pi(y)\right|&\le e^{-\gamma t}+o(L^{-dl}).
\label{eq:estim_cont2}
 \end{align} 
\end{pro}

\begin{proof}
Recall that given a discrete time random walk on a finite set,
with transition matrix $P$ and reversible measure the uniform measure
$\pi$, a result of Sinclair and Jerrum \cite{cf:SinJer89} gives an estimate of the speed of
convergence to equilibrium. Indeed in this case
$P$ has all real eigenvalues, namely $1=\lambda_0>\lambda_1\ge\cdots\ge
\lambda_{n-1}$. Let $\lambda=\max\{|\lambda_i|:i=1,\ldots,n-1\}$. It is well known that
$\lambda<1$.
Then for all $t\in\N_0:=\{n \in \mathbb{Z}: n\geq 0\}$
\[
 \max_{x,y}\left|p^{(t)}(x,y)-\pi(y)\right|\le
{\lambda^{t}}
\le {\exp(-(1-\lambda)t)},
\]
where $p^{(t)}(x,y)$ is a $t$-step probability of the walk.
If we are able to estimate $\lambda$ we are done. 
If $\lambda=\lambda_1$
then the following 
(which is known as Cheeger's inequality (see \cite[Theorem 6.2.1]{cf:Durrett}) is useful
\begin{equation*}
 \frac12\iota^2\left(\min_{x,y:p(x,y)>0}p(x,y)\right)^2\le1-\lambda_1.
\end{equation*}
A sufficient condition for $\lambda=\lambda_1$ is that all the eigenvalues are positive, 
which for instance holds when we consider a lazy random walk, that is one which stays put 
with probability at least 1/2. 
It is thus clear that for any small world $S$ in $Q_\alpha^L$, 
if the random walk $\Xd_t$ 
is such that
 $\lambda=\lambda_1$, then
\begin{equation}\label{eq:q-estim-conveq}
 \max_{x,y}\left|\Pd_S^x(\Xd_t=y)-\pi(y)\right|\le
 {\exp(-c\alpha^2t)},
\end{equation}
where $c=\frac12\left(\min_{x,y:p(x,y)>0}p(x,y)\right)^2$ depends only on 
$\Delta$ and is strictly positive (recall that $\Delta$ is adapted and translation
invariant on $\Lambda(L)$).
Moreover by Proposition \ref{th:isorrg} 
\begin{align}
  \max_{x,y}\left|\Pd^x(\Xd_t=y)-\pi(y)\right|&\le\sum_S
  \mathbf P(S)\max_{x,y}\left|\Pd_S^x(\Xd_t=y)-\pi(y)\right|\nonumber \\
&\le \exp(-c\alpha^2t)\mathbf P(Q_\alpha^L
)+2\mathbf P(Q_\alpha^L
)\nonumber \\
&\le\exp(-c\alpha^2
t)+o(L^{-dl})
\label{eq:q-estim-conveq2}.
 \end{align}
It is easy starting from \eqref{Xcont} to prove that \eqref{eq:q-estim-conveq} and 
\eqref{eq:q-estim-conveq2} still hold in continuous time with a different 
constant in the exponential. Namely, one has to replace 
$c\alpha^2$ with $1-\exp(-c\alpha^2)$.

We are left with the proof that \eqref{eq:q-estim-conveq} and \eqref{eq:q-estim-conveq2}
 hold for any random walk (not just for the lazy one) with different constants.
 This can be proven by coupling $\Xd$ with
a random walk $\Yd=\{\Yd_t\}_{t\ge0}$ with transition matrix $P^\prime$ such that 
$p^\prime(x,x)=(1+p(x,x))/2$, $p^\prime(x,y)=p(x,y)/2$: the process $\Yd$ is ``lazy''
and moves with $\Xd$ when 
a Bernoulli random variable with parameter $1/2$ equals 1, otherwise it stays put.
We leave the computation to the reader.
\end{proof}

\section{Laplace transform estimates}
\label{hitting_time}

Let $\Tc_L=\inf\{s>0:\Xc_s=\Yc_s\}$ (respectively $\Td_L$) be the first time, after time 0, that two independent continuous
 (respectively discrete) time random walks $\Xc_t$ and $\Yc_t$ on the random graph $\mathcal{S}$ meet.
Clearly the law of $\Tc_L$ (with respect to either $\Pc_S$ or $\Pc$) 
depends on the starting sites of the walkers.
Without loss of generality, we assume that $\Yc_0=0$ and $\Xc_0=x$
(if we need to stress the dependence on $L$, we write $\Xc_0=x_L$).\\
We introduce the following (annealed) Laplace transforms in continuous time,
\[\begin{split}
\Gc^L(x,\lambda)&:=\int_{0}^{\infty}e^{-\lambda t}\Pc^{x,0}(\Xc_{t}=\Yc_{t})dt=
\int_{0}^{\infty}e^{-\lambda t}\Pc^{x}(\Xc_{2t}=0)dt,\\
\Fc^L(x,\lambda)&:=\int_{0}^{\infty}e^{-\lambda t}\Pc^{x,0}(\Tc_L\in dt)
,
 \end{split}
\] 
where $\Pc^{x,0}$ denotes the product law of the two walkers. The corresponding quenched transforms are, given $S\in \mathcal S^L(\Omega)$,
\begin{align*}
\Gc^L_S(x,\lambda):=\int_{0}^{\infty}e^{-\lambda t}\Pc_S^{x,0}(\Xc_{t}=\Yc_{t})dt,\qquad 
\Fc^L_S(x,\lambda):=\int_{0}^{\infty}e^{-\lambda t}\Pc_S^{x,0}(\Tc_L\in dt).
 \end{align*}
We are interested in the asymptotic behaviour, as $L\to\infty$, of $T_L/(2L)^d$,
thus we study the previous Laplace transforms with parameter $\lambda/(2L)^d$.\\
The discrete time version of such Laplace transforms are defined in a similar way, 
but the integrals are replaced by  sums. With a slight abuse of notation we omit the 
superscript $\sim$ on the discrete time random walk when not necessary and we use 
$X_{t}$, $Y_{t}$, $T_{L}$, $\Prob^{\mu}_{S}$, $G^L_S(x,\lambda)$ and $F^L_S(x,\lambda)$ 
both in discrete and continuos time version of the process: since the proofs are similar, 
we detail the latter one and we only point out the differences. 
\subsection{Estimates for G}
\label{sec:Gcont}

We first note that the evaluation of the limit of the annealed transforms can
be done considering only small worlds with large isoperimetric constants,
that is on $Q^L_\alpha$ (which was defined by \eqref{QL}).
Let $\K:=\{K\subset \mathbb{R}: \inf K >0\}$.
\begin{lem}\label{th:qisocont} 
Let
\[\begin{split}
   g_{L}&:=\sum_{S \in (Q_{\alpha}^L)^c}\mathbf P({S})
	\int_{0}^{\infty}e^{- \frac{\lambda t}{(2L)^{d}}}
	\mathbb{P}_{S}^{x}(X_{2t}=0)dt,\\
  f_L&:=\sum_{S \in (Q_{\alpha}^L)^c}\mathbf P({S})
	\int_{0}^{\infty}e^{- \frac{\lambda t}{(2L)^{d}}}
	\mathbb{P}_{S}^{x,0}(T_L\in dt).
  \end{split}
\]
There exists $\alpha>0$ such that $\mathbf{P}(Q_{\alpha}^L)\stackrel{L\to\infty}{\to}1$,
$g_L\stackrel{L\to\infty}{\to}0$ and $f_L\stackrel{L\to\infty}{\to}0$
(for each $K \in \K$, uniformly for $\lambda \in K$).
\end{lem}
\begin{proof}
By Proposition~\ref{th:isosw} we may choose $\alpha$ such that
$\mathbf P((Q^L_\alpha)^c)=o(L^{-2d})$. Then
\[
0\le f_L\le g_L\le\mathbf P((Q_{\alpha}^L)^c)\int_{0}^\infty e^{- \frac{\lambda t}{(2L)^{d}}}dt
=\mathbf P((Q_{\alpha}^L)^c)\frac{(2L)^d}{\lambda}
\stackrel{L\to\infty}{\to}0.
\]
\end{proof}
The limit of the sum defining $G$, from $\log \log L$ 
to infinity does not depend
on the sequence of small worlds, provided that they are chosen with large isoperimetric
constant. From now on, if not otherwise stated, we write $t_L=\log\log L$, fix $\alpha$ 
such that $\mathbf P((Q^L_\alpha)^c)=o(L^{-2d})$ 
and write $Q^L$ instead of $Q^L_\alpha$. 
\begin{lem}\label{th:tlinfcont}
If for all $L$ we choose $S \in Q^L$ 
and $x_L\in\Lambda(L)$, then for all $\lambda>0$
$$\lim_{L \to \infty}\int_{t_L}^{\infty}
e^{-\frac{\lambda t}{(2L)^{d}}}\mathbb{P}^{x_L}_{S}(X_{2t}=0)dt=\frac{1}{\lambda}.
$$ 
Moreover, the convergence is uniform with respect to the choice of the sequences
$S \in Q^L$, $x_L\in\Lambda(L)$ (and of $\lambda$).
\end{lem}
\begin{proof}
Note that
\begin{multline}\label{eq:G30}
\int_{t_{L}}^{\infty}e^{-\frac{\lambda t}{(2L)^{d}}}\mathbb{P}_{S}^{x_L}(X_{2t}=0)dt\\
=\int_{t_{L}}^{\infty}e^{-\frac{\lambda t}{(2L)^{d}}}\frac{1}{(2L)^{d}}dt+
\int_{t_{L}}^{\infty}e^{-\frac{\lambda t}{(2L)^{d}}}\left(\mathbb{P}_{S}^{x_L}(X_{2t}=0)-
\frac{1}{(2L)^{d}}\right)dt.
\end{multline}
The limit of the first term is uniform in $\lambda$ and it is $1/\lambda$.
Since $S$ is chosen in $Q^L$, by \eqref{eq:estim_cont1} there exists a positive
constant $\gamma$ 
(recall that $\gamma$ depends on $\alpha$ which is now fixed)
such that the second sum on the right hand side of \eqref{eq:G30} is smaller or equal to
\[\begin{split}
\int_{t_{L}}^{\infty}e^{-\frac{\lambda t}{(2L)^d}}e^{-\gamma 2t}dt=
\frac{e^{-(\lambda/(2L)^d+2\gamma)t_L}}{\lambda/(2L)^d+2\gamma},
\end{split}
\]
which tends to 0 as $L$ goes to infinity (uniformly with respect to all the choices of the statement).
\end{proof}
Recall that, given a vertex $x\in\Lambda(L)$, there is a unique
vertex $+(x)$ in the big world (if $x=0$ we write 0 instead of $+(0)$)
and that by $I(0,t)$ we denote the set of small worlds which look
like the big world in a ball of radius $t$ around 0 (see Definition~\ref{def:I(x,t)}).
We now prove that, if $L$ is sufficiently large,
for a wide choice of $S$ (i.e. $S$ in a set with $\mathbf P$-probability which tends to 1 
as $L$ increases to infinity),
we have that 
 $G^L_S(x_L,\lambda/(2L)^d)$ is close to $1/\lambda +\Gp_{\mathcal{B}}(x_L)$.
\begin{theorem}\label{th:Gquenchedcont}
Let
\[
h^L_S(\lambda)=\left|G^L_S(x_L,\lambda/(2L)^d) 
- \frac1\lambda - \Gp_{\mathcal{B}}(x_L)\right|.
\]
For all $\eps>0$ there exists $\widetilde L$ such that for all $\lambda$, $x_L$ and $L\ge\widetilde L$
we have that $Q^L\cap I(0,t^2_L)\subset (S:h^L_S(\lambda)\le\eps)$. 
If $d_S(0,x_L)>t^2_L$ then $Q^L\subset (S:h^L_S(\lambda)\le\eps)$.
\end{theorem}
\begin{proof}
Note that
\begin{align*}
 h^L_S(\lambda)&\le 
\left|\int_{t_L}^\infty e^{-\frac{\lambda t}{(2L)^d}}
\mathbb{P}_S^{x_L}(X_{2t}=0)dt-\frac1\lambda
\right|\\
&+\left|\int_{0}^{t_L} e^{-\frac{\lambda t}{(2L)^d}}\left(
\mathbb{P}_S^{x_L}(X_{2t}=0)dt-\mathbb{P}^{+(x_L)}_{\mathcal{B}}(X_{2t}=0)\right)dt\right|\\
&+\int_{t_{L}}^{\infty}\mathbb{P}_{\mathcal{B}}^{+(x_L)}(X_{2t}=0)dt
+\int_{0}^{t_L}(1- e^{-\frac{\lambda t}{(2L)^d}})
\mathbb{P}_{\mathcal{B}}^{+(x_L)}(X_{2t}=0)dt. 
\end{align*}
By Lemma \ref{th:tlinfcont}, if $S\in Q^L$, the first term is smaller than $\eps/4$ provided
that $L$ is large.

Since either $S\in I(0,t^{2}_L)$ or $d_S(0,x_L)>t^{2}_L$, the probabilities of a meeting 
before time $t_L$ on $S$ and on the big world differ only if 
the value of the underlying Poisson process $N_{t}$ (recall equation \eqref{Xcont})
at time $2t_L$
is at least $t^{2}_{L}$: by Chebyshev's inequality the second term of the 
right hand side is smaller than  
$$
2t_L\Prob (N_{2t_L}\geq t^2_{L})\leq 
 \frac{(2t_L)^2}{(t_L^2-t_L)^2}\leq \eps/4 $$
if $L$ is large enough.

Note that by \eqref{eq:Pbwubcont} the integrands of the last two terms are both
dominated by $\mathbb{P}_{\mathcal{B}}^{0}(X_{2t}=0)$ which does not depend
on $L$ and is integrable. Thus by the Dominated Convergence Theorem
they are both smaller than  $\eps/4$
if $L$ is sufficiently large.
\end{proof}

\begin{theorem}\label{th:Ganncont}
For all $K \in \K$, $\eps>0$ there exists $\widetilde L$ such that for all $L\ge\widetilde L$,
 $x_L\in \Lambda(L)$, and $\lambda\in K$, 
\[
\left|G^L(x_L,\lambda/(2L)^d) 
- \frac1\lambda-\Gp_{\mathcal{B}}(x_L) \right|\le\eps.
\]
\end{theorem}

\begin{proof}
Recall that 
\[
 G^L(x_L,\lambda/(2L)^d) =\sum_S\mathbf P(S) G^L_S(x_L,\lambda/(2L)^d) .
\]
By Theorem \ref{th:Gquenchedcont} there exists $\widetilde L$ such that for all $L\ge\widetilde L$,
\[
 \left|\sum_{S\in Q^L\cap I(0,t^2_L)}\mathbf P(S) G^L_S(x_L,\lambda/(2L)^d)
- \frac1\lambda - \Gp_{\mathcal{B}}(x_L)\right|\le\eps/3.
\]
Thus, since $\mathbf P((Q^L)^c)$ and $\mathbf P(I(0,t^2_L)^c)$ are both small
if $L$ is large, we may choose $\widetilde L$ such that for all $\lambda\in K$
and $L\ge \widetilde L$
\[
 \sum_{S\in (Q^L)^c\cup (I(0,t^2_L))^c}\mathbf P(S)\left(\frac1\lambda + 
\Gp_{\mathcal{B}}(x_L)\right)\le\eps/3.
\]
Now we only need to prove that
\[
 \sum_{S\in (Q^L)^c\cup I(0,t^2_L)^c}\mathbf P(S) G^L_S(x_L,\lambda/(2L)^d)
\le\eps/3.
\]
By Lemma \ref{th:qisocont} we know that $\sum_{S\in (Q^L)^c}\mathbf P(S) G^L_S(x_L,\lambda/(2L)^d)
\le\eps/6$ for all $L\ge\widetilde L$ and $\lambda> 0$.
Finally, by Proposition \ref{th:xgood} and Lemma \ref{th:tlinfcont}, for some $C>0$ and $L$
sufficiently large
\[\begin{split}
    &\sum_{S\in Q^L\cap I(0,t^2_L)^c}\mathbf P(S) G^L_S(x_L,\lambda/(2L)^d) \\
& =\sum_{S\in Q^L\cap I(0,t^2_L)^c}\mathbf P(S)\left(\int_{0}^{t_L} e^{-\frac{\lambda t}{(2L)^d}}
\mathbb{P}_S^{x_L}(X_{2t}=0)+\int_{t_L}^\infty e^{-\frac{\lambda t}{(2L)^d}}
\mathbb{P}_S^{x_L}(X_{2t}=0)\right)\\
&\le \left(t_L+\frac1\lambda+C\right)\mathbf P(I(0,t^2_L)^c)\le\eps/6.
 \end{split}
\]
\end{proof}

\subsection{From G to F}
\label{sec:GFcont}

We note that if $x_{L}\neq 0$ then $G^L_S(x_L,{\lambda}/{(2L)^{d}})$ may be written as
\begin{equation*}
\sum_z\int_{0}^\infty e^{-\frac{\lambda q}{(2L)^d}}
\Prob_S^{z}(X_{2q}=z)dq 
\int_{0}^\infty
e^{-\frac{\lambda s}{(2L)^d}}
{\Prob_S^{x_L,0}}(T_L\in ds,X_s=z).
\end{equation*}
while $G^L_S\left(0,
{\lambda}/{(2L)^{d}}\right)$ is equal to
\begin{equation*}
1
+\sum_z\int_{0}^\infty e^{-\frac{\lambda q}{(2L)^d}}
\Prob_S^{z}(X_{2q}=z)dq
\int_{0}^\infty
e^{-\frac{\lambda s}{(2L)^d}}
{\Prob_S^{0,0}}(T_L\in ds,X_s=z).
\end{equation*}
Define $H_1$, $H_2$ and $H_3$ (which depend on $S$, $x_L$ and $L$) by
\[\begin{split}
   H_1&:=\sum_z\int_{0}^{t_L} e^{-\frac{\lambda q}{(2L)^d}}
\Prob_S^{z}(X_{2q}=z)dq
\int_{0}^{t_L}
e^{-\frac{\lambda s}{(2L)^d}}
{\Prob_S^{x_L,0}}(T_L\in ds,X_s=z)\\
H_2&:=\sum_z\int_{0}^{t_L} e^{-\frac{\lambda q}{(2L)^d}}
\Prob_S^{z}(X_{2q}=z)dq
\int_{t_L}^\infty
e^{-\frac{\lambda s}{(2L)^d}}
{\Prob_S^{x_L,0}}(T_L\in ds,X_s=z)\\
H_3&:=\sum_z\int_{t_L}^\infty e^{-\frac{\lambda q}{(2L)^d}}
\Prob_S^{z}(X_{2q}=z)dq
\int_{0}^\infty
e^{-\frac{\lambda s}{(2L)^d}}
{\Prob_S^{x_L,0}}(T_L\in ds,X_s=z).
  \end{split}
\]
By Lemma \ref{th:qisocont}, for all $L$ sufficiently large and if
the limit exists,
\begin{equation}\label{eq:GLQLcont}
 \lim_{L\to\infty}G^L\left(x_L,\frac{\lambda}{(2L)^{d}}\right)=
\lim_{L\to\infty}\sum_{S\in Q^L}\mathbf P(S)(H_1+H_2+H_3).
\end{equation}
Clearly if $x_L=0$ for all $L$ sufficiently large we only need to add 1 to the previous limit. The same equality holds in discrete time, replacing the integral with the sum.\\
We now study each of the three summands separately, in order to obtain
the limit of $F^L$ as a function of the limit of $G^L$.
\begin{lem}\label{th:H1cont}
 If $S\in I(0,t^{2}_L)$ and $x_L\in\Lambda(L)$, for each $\eps >0$ there exists $\widetilde{L}$ such that for each $L>\widetilde{L}$ then
\begin{equation}\label{eq:thH11cont}
  \left|H_1-\int_{0}^{t_L} e^{-\frac{\lambda q}{(2L)^d}}
\Prob_{\mathcal{B}}^{0}(X_{2q}=0)dq
\int_{0}^{t_L}
e^{-\frac{\lambda s}{(2L)^d}}
{\Prob_S^{x_L,0}}(T_L\in ds)\right|<\eps.
\end{equation}
This inequality also holds whenever $d_S(0,x_L)>t^2_L$.
Moreover, uniformly with respect to the choice of the sequence $\{x_L\}_L$ and of $\lambda$,
\begin{multline}\label{eq:thH12cont}
\sum_{S\in (I(0,t^{2}_L))^c}\mathbf P(S)
\sum_z\int_{0}^{t_L} e^{-\frac{\lambda q}{(2L)^d}}
\Prob_S^{z}(X_{2q}=z)dq\\
\int_{0}^\infty
e^{-\frac{\lambda s}{(2L)^d}}
{\Prob_S^{x_L,0}}(T_L\in ds,X_s=z)\stackrel{L\to \infty}{\rightarrow} 0.
\end{multline}
\end{lem}

\begin{proof}
We first note that if we consider the discrete time random walk
then at time $t_L$ the walker is at a distance at most $t_L$
from her starting site. Thus the sites $z$ in the sum of $H_1$
are those at distance at most $t_L$ from 0 (all other terms
being zero). Then, since $q\le t_L$ and $S\in I(0,t_L^2)$ 
(actually, in discrete time $S\in I(0,2t_L)$ suffices),
we have that  
$\Prob_S^{z}(X_{2q}=z)=\Prob_{\mathcal{B}}^{0}(X_{2q}=0)$
and the difference in equation~\eqref{eq:thH11cont}
is equal to zero. On the other hand, if  $d_S(0,x_L)>t^2_L$
it is not possible for two random walkers starting at 0 and at $x_L$
respectively, to meet within time $t_L$ and all quantities in equation~\eqref{eq:thH11cont}
are zero.

In continuous time the walkers may take a large number of steps even
in a small amount of time (though this is quite unlikely).
Denote by $N_t$ the Poisson process underlying the random walk from 0
(see the second integral in $H_1$) and by $N_t^\prime$ 
the Poisson process underlying the random walk from $z$
(see the first integral in $H_1$).

Suppose that $S\in I(0,t_L^2)$: if 
$(N_{t_L}<t_L^2/2)$ and $(N_{2t_L}^\prime<t_L^2)$
then the whole path from $z$ to $z$ of duration $2q$
lies in the ball of radius $t_L^2$ centered at 0
and the law of the walk coincides with the corresponding
walk on the big world.
By Chebyshev's inequality 
\begin{equation}\label{eq:Poisson-Cheb}
 \Prob(N_{t_L}\ge t_L^2/2)\le\frac{C}{t_L^3}\text{ and }
\Prob(N_{2t_L}^\prime\ge t_L^2)\le\frac{C}{t_L^3},
\end{equation}
for some positive constant $C$.
Then straightforward computation shows that $H_1$ differs from
\begin{equation}
 \sum_z\int_{0}^{t_L} e^{-\frac{\lambda q}{(2L)^d}}
\Prob_S^{z}(X_{2q}=z,N^\prime_{2t_L}<t_L^2)dq
\int_{0}^{t_L}
e^{-\frac{\lambda s}{(2L)^d}}
{\Prob_S^{x_L,0}}(T_L\in ds,X_s=z,N_{t_L}<t_L^2/2)
\end{equation}
at most by $3t_L^2\cdot C/t_L^3$.
The same argument that we used in discrete time gives
\[
\Prob_S^{z}(X_{2q}=z,N^\prime_{2t_L}<t_L^2)=\Prob_{\mathcal{B}}^{0}(X_{2q}=0,N^\prime_{2t_L}<t_L^2)
\]
and by equation~\eqref{eq:Poisson-Cheb}, equation~\eqref{eq:thH11cont} follows
(similarly one proves it in the case $d_S(0,x_L)>t^2_L$).

In order to prove \eqref{eq:thH12cont}, note that for some $C>0$
\begin{equation*}
\begin{split}
 \sum_{S\in I(0,t^2_L)^c}&\mathbf P(S) 
\sum_z\int_{0}^{t_L} e^{-\frac{\lambda q}{(2L)^d}}
\Prob_S^{z}(X_{2q}=z)
\int_{0}^{\infty}
e^{-\frac{\lambda s}{(2L)^d}}
{\Prob_S^{x_L,0}}(T_L=s,X_s=z)\\
& \le Ct_L\,\mathbf P(I(0,t^2_L)^c)
\,F^L(x_L,\lambda/(2L)^d),
\end{split}
\end{equation*}
which, by Proposition \ref{th:xgood} and since $F^L(x,\lambda)\le1$
for all $\lambda$ and $x$, goes to 0, uniformly in $x_L$ and $\lambda$, as $L$ goes to infinity.

\end{proof}

\begin{lem}\label{th:H2cont}
For all $K\in \K$, $\eps>0$ there exists $\widetilde L$ such that for all $L\ge\widetilde L$, $x_L$
and $\lambda\ \in K$,
\begin{multline}\label{eq:H2cont}
\left|\sum_{S\in Q^L}\mathbf P(S)H_2
-\int_{0}^{t_L} e^{-\frac{\lambda q}{(2L)^d}}
\Prob_{\mathcal{B}}^{0}(X_{2q}=0)dq\right.\\
\left.\sum_{S\in Q^L}\mathbf P(S)\int_{t_L}^\infty
e^{-\frac{\lambda s}{(2L)^d}}
{\Prob_S^{x_L,0}}(T_L\in ds)\right|\le\eps.
\end{multline}
\end{lem}

\begin{proof}
Note that $\sum_{S\in Q^L}\mathbf P(S)H_2$ can be written as
\[\begin{split}
&\sum_{z}\sum_{S \in Q^L\cap I(z,t_L^{2})}
\mathbf P(S)\int_{0}^{t_{L}}e^{-\frac{\lambda 
q}{(2L)^{d}}}\mathbb{P}_{S}^{z}(X_{2q}=z)dq\int_{t_{L}}^{\infty}e^{-\frac{\lambda s}{(2L)^{d}}}
\mathbb{P}_{S}^{x_L,0}(T_L\in ds,X_{s}=z)\\
\\
&+\sum_{z}\sum_{S \in Q^L\cap I(z,t_L^{2})^c}\mathbf P(S)
\int_{0}^{t_{L}}e^{-\frac{\lambda 
q}{(2L)^{d}}}\mathbb{P}_{S}^{z}(X_{2q}=z)dq\int_{t_{L}}^{\infty}e^{-\frac{\lambda s}{(2L)^{d}}}
\mathbb{P}_{S}^{x_L,0}(T_L\in ds,X_{s}=z)\\
&=H_{2,1}+H_{2,2}.
\end{split}
\]
We prove that $H_{2,2}\to0$, indeed since $\exp(-\lambda q/(2L)^d)\mathbb P_S^{z}(X_{2q}=z)\le 1$ then
\[\begin{split}
H_{2,2}&\le t_L\sum_z\sum_{S \in Q^L\cap I(z,t^{2}_L)^c}\mathbf P(S)
\left\{\int_{t_{L}}^{\log L}e^{-\frac{\lambda s}{(2L)^{d}}}
\mathbb{P}_{S}^{x_L,0}(X_{s}=Y_{s}=z)ds\right.\\
&\left.+\int_{\log L}^\infty e^{-\frac{\lambda s}{(2L)^{d}}}
\mathbb{P}_{S}^{x_L,0}(X_{s}=Y_{s}=z)ds\right\}
=:H_{2,2,1}+H_{2,2,2}.
\end{split}
\]
Note that $H_{2,2,1}$ is smaller or equal to
\[\begin{split}
&t_L\sum_{S \in Q^L}\mathbf P(S)
\int_{t_{L}}^{\log L} e^{-\frac{\lambda s}{(2L)^{d}}}
\Prob_S^{0}(X_{2s}=x_L)ds.
\end{split}
\]
We write $\Prob_S^{0}(X_{2s}=x_L)\le \left|\Prob_S^{0}(X_{2s}=x_L)-1/(2L)^d\right|
+1/(2L)^d$, which by \eqref{eq:estim_cont1} is smaller or
equal to $e^{-\gamma s}+1/(2L)^d$ (recall that $\gamma$ depends only on the parameter
$\alpha$ which has been fixed in $Q^L$).
It is thus only a matter of computation to show that  $H_{2,2,1}$
goes to zero (uniformly in $x_L$ and $\lambda$) as $L$ goes to infinity.\\
Now we consider $H_{2,2,2}$.
Note that 
$\mathbb{P}_{S}^{x_L,0}(X_{s}=Y_{s}=z)=\Prob_S^z(X_s=0)\Prob_S^z(Y_s=x_L)$. 
Write 
\[
 \Prob_S^z(X_s=0)=\Prob_S^z(X_s=0)-\frac{1}{(2L)^d}+\frac{1}{(2L)^d}
\]
and do the same for $\Prob_S^z(Y_s=x_L)$. Using \eqref{eq:estim_cont1} and 
Proposition \ref{th:xgood}, we have that
$H_{2,2,2}$ is smaller or equal to
\[\begin{split}
&t_L\sum_z\sum_{S \in Q^L\cap I(z,t^2_L)^c}\mathbf P(S)
\int_{\log L}^\infty e^{-\frac{\lambda
 s}{(2L)^{d}}}\left(e^{-\gamma s}+\frac{1}{(2L)^d}\right)^2ds\\
&\le C t_L L^d\frac{M^{4t^2_L}}{L^d}
\int_{\log L}^{\infty} e^{-\frac{\lambda
 s}{(2L)^{d}}}\left(e^{-2\gamma s}+\frac{1}{(2L)^{2d}}+\frac{2e^{-\gamma s}}{(2L)^{d}}\right)ds
\le C^\prime \frac{t_LM^{4t_L^2}}{L^{2\gamma}},\\
\end{split}
\]
for some $C^\prime>0$. The last quantity 
goes to 0 (uniformly in $x_L$ and $\lambda \in K$ for each $K\in \K$) as $L\to\infty$
(the numerator contains only terms which are logarithmic in $L$).

We now prove that equation \eqref{eq:H2cont} holds with $H_{2,1}$
in place of $\sum_{S\in Q^L}\mathbf P(S)H_2$.
The same arguments that we used to prove that $H_{2,2}$ converges to 0 tell us that
\[\begin{split}
\sum_{z}\sum_{S \in Q^L\cap I(z,t^{2}_L)^c}\mathbf P(S)\int_{0}^{t_{L}}e^{-\frac{\lambda 
q}{(2L)^{d}}}\Prob_{\mathcal{B}}^{0}(X_{2q}=0)dq
\int_{t_{L}}^{\infty}
e^{-\frac{\lambda s}{(2L)^{d}}}\mathbb{P}_{S}^{x_L,0}(T_L \in ds,X_{s}=z).
\end{split}
\]
converges to zero.
Thus we are left with
\begin{equation}\label{eq:lastpiece}
 \begin{split}
 \sum_z\sum_{S \in Q^L\cap I(z,t^{2}_L)}\mathbf P(S)\int_{0}^{t_{L}}e^{-\frac{\lambda 
q}{(2L)^{d}}}&\left|\mathbb{P}_{S}^{z}(X_{2q}=z)
-\Prob_{\mathcal{B}}^{0}(X_{2q}=0)\right|dq\\
&\int_{t_{L}}^{\infty}e^{-\frac{\lambda s}{(2L)^{d}}}
\mathbb{P}_{S}^{x_L,0}(T_L\in ds,X_{s}=z)
\end{split}
\end{equation}
and we have to prove that it is small when $L$ is large.
By \eqref{eq:Poisson-Cheb}
 we have $\int_0^{t_L}|\Prob_S^{z}(X_{2q}=z)-\Prob_{\mathcal{B}}^{0}(X_{2q}=0)|dq\leq C/t_L^2$. 
Thus the quantity in equation \eqref{eq:lastpiece} is at most
\[
 \frac{C}{t_L^2}\sum_{S\in Q^L}\int_0^\infty e^{-\frac{\lambda s}{(2L)^{d}}}
\mathbb{P}_S^{x_L,0}(T_L \in ds)\le \frac{C}{t_L^2}F^L(x_{L},\lambda/(2L)^{d})\]
which can be taken as small as we want if $L$ is large (recall that 
$F^L(x_{L},\lambda/(2L)^{d})\le1$ and that all sums and integrals are interchangeable
since all quantities are nonnegative). 
\end{proof}

\begin{lem}\label{th:H2quenchedcont}
Let
\begin{equation}
a^L_\lambda(S):=H_2
-\int_{0}^{t_L} e^{-\frac{\lambda q}{(2L)^d}}
\Prob_{\mathcal{B}}^{0}(X_{2q}=0)dq
\int_{t_L}^\infty
e^{-\frac{\lambda s}{(2L)^d}}
{\Prob_S^{x_L,0}}(T_L\in ds).
\end{equation}
Then $a^L_\lambda\ge0$ and $a^L_\lambda\to 0$ in probability (for each $K \in \K$, uniformly in $x_L$ and $\lambda\in K$), that is for all $K \in \K$, $\eps>0$
and $\delta>0$ there exists $\widetilde L$ such that for all $L\ge\widetilde L$ and $x_L$
\[
 \mathbf P(A^L_\eps(K)):=\mathbf P(S:a^L_\lambda(S)\leq \eps, \forall \lambda \in K)\geq 1-\delta.
\]
\end{lem}
\begin{proof}  
We first note that for all $z$ and $S$,
$\Prob_{S}^{z}(X_{2q}=z)\ge\Prob_{\mathcal{B}}^{0}(X_{2q}=0)$,
hence $a^L_\lambda(S)\ge0$. Suppose by contradiction that there
exist $K$, $\eps>0$ and $\delta>0$ such that $\mathbf P(A^L_\eps(K))\le 1-\delta$
infinitely often.
Then infinitely often
\[
 \sum_S\mathbf P(S) a^L_\lambda(S)>\delta\eps.
\]
By Lemmas \ref{th:H2cont} and \ref{th:qisocont}, there exists $\widetilde{L}$ such that $\sum_S\mathbf P(S) a_\lambda^L(S)<\delta\eps$ for each $L\geq \widetilde{L}$, $x_{L}$, $\lambda \in K$, whence the contradiction.
\end{proof}

\begin{lem}\label{th:H3cont}
For all $K\in \K$ and $\eps>0$ there exists $\widetilde L$ such that for all $L\ge\widetilde L$, $S\in Q^L$, 
$x_L$ and $\lambda \in K$,
\begin{equation}\label{eq:limH3cont}
\left|H_3
-\frac1\lambda 
F^L_S(x_L,\lambda /(2L)^d)\right|\le\eps.
\end{equation}
\end{lem}
\begin{proof}
Note that
\[\begin{split}
    H_3=&\sum_z\int_{t_L}^\infty \left(\Prob_S^{z}(X_{2q}=z)-\frac{1}{(2L)^d}\right)
e^{-\frac{\lambda q}{(2L)^d}}dq
\int_{0}^\infty
e^{-\frac{\lambda s}{(2L)^d}}
{\Prob_S^{x_L,0}}(T_L\in ds,X_s=z)\\
&+
\int_{t_L}^\infty e^{-\frac{\lambda q}{(2L)^d}}\frac1{(2L)^d}
F^L_S(x_L,\lambda /(2L)^d)dq,
  \end{split}
\]
and the modulus of the first member does not exceed
\[
\int_{t_L}^\infty e^{-\frac{\lambda q}{(2L)^d}-2\gamma q}dq
\int_{0}^\infty
e^{-\frac{\lambda s}{(2L)^d}}
{\Prob_S^{x_L,0}}(T_L\in ds)\le
C \exp(-\gamma t_L)
\]
by \eqref{eq:estim_cont1} (recall that $S\in Q^L$ and the fact that $F^L_S(x_L,\lambda /(2L)^d)\le1$.
The claim follows since for $L$ sufficiently large
\[
 \left|\frac1{(2L)^d}\int_{t_L}^\infty e^{-\frac{\lambda q}{(2L)^d}}dq-\frac1\lambda\right|<\eps/2.
\]
\end{proof}

\begin{theorem}\label{th:Fquenchedcont}
Let 
$$
b_\lambda^L(S):=\left|F^L_S\left(x_L,\frac{\lambda}{(2L)^{d}}\right)-
\frac{\Gp_{\mathcal{B}}(x_L)+\frac{1}{\lambda}-\I_{\{0\}}(+(x_L))}{\Gp_{\mathcal{B}}(0)+
\frac{1}{\lambda}}\right|.
$$

(a) Then $b_\lambda^L\to 0$ in probability, for each $K \in \K$ such that $\sup K 
<\infty$, uniformly in 
$x_L\in\Lambda(L)$ and $\lambda \in K$, namely for all $\eps>0$ 
$(S:b_\lambda^L(S)\le\eps, \forall \lambda \in K)\supset Q^L\cap I(0,t_{L}^{2})\cap A^L_{\eps/2}(K)$ for all $L$ sufficiently
large ($A^L_\eps(K)$ was defined in Lemma~\ref{th:H2quenchedcont}).

(b) For all $\eps>0$, $(S:b_\lambda^L(S)\le\eps, \forall \lambda \in K)\supset Q^L\cap (S:d_{S}(0,x_{L})>t^{2}_L)\cap A^L_{\eps/2}(K)$ for all $L$ sufficiently large.
\end{theorem}
\begin{proof}  
\textit{(a)} Note that by the Dominated Convergence Theorem (recall \eqref{eq:GBevcont}),
since $\sup K=\lambda_0<\infty$, 
for each $\eps >0$ and $L>\widetilde{L}$ large enough  
\begin{align}
\left|\int_{0}^{t_L} e^{-\frac{\lambda q}{(2L)^d}}\Prob_{\mathcal{B}}^{0}(X_{2q}=0)dq-\Gp_{\mathcal{B}}(0)\right|<\eps.
\label{eq:tmp}
\end{align}
Consider
\[
G^L_S(x_L,\lambda/(2L)^d)-\I_{\{0\}}(x_L)-\left(\Gp_{\mathcal{B}}(0)+\frac{1}{\lambda}\right)
F^L_S(x_L,\lambda/(2L)^d).
\]
Writing $G^L_S(x_L,\lambda/(2L)^d)=\I_{\{0\}}(x_L)+H_1+H_2+H_3$, using
 Lemmas \ref{th:H1cont}, \ref{th:H2quenchedcont}, \ref{th:H3cont} and \eqref{eq:tmp} 
follows that the previous difference is smaller than $\eps$ when $L$ is sufficiently large and
 $S\in Q^L\cap I(0,t^2_L)\cap A^L_{\eps/2}(K)$. Note that all these three sets have probability
which converges to 1. By Theorem \ref{th:Gquenchedcont} we conclude that $b_\lambda^L$
goes to zero in probability.

\textit{(b)}  Since equation \eqref{eq:thH11cont} holds also when
$d_{S}(0,x_L)>t^2_{L}$, we have that
 $S \in Q^L\cap (S:d_{S}(0,x_L)>t^2_{L})\cap A^L_{\eps/2}(K)$. 
\end{proof}

\begin{theorem}\label{th:Fanncont}
For each $\eps>0$ and $K \in \K$ such that $\sup K<\infty$,
there exists $\widetilde{L}$ such that for each $L>\widetilde{L}$, $\lambda \in K$ and for each sequence $\{x_L\}_L$  such that $x_L\in \Lambda(L)$,
$$
\left|F^L\left(x_L,\frac{\lambda}{(2L)^{d}}\right)-
\frac{\Gp_{\mathcal{B}}(x_L)+\frac{1}{\lambda}-\mathbf 1_{\{0\}}(+(x_L))}{\Gp_{\mathcal{B}}(0)
+\frac{1}{\lambda}}\right|\leq \eps.
$$
\end{theorem}
\begin{proof}
To keep notation simple we deal only with the case $+(x_L)\neq0$ (the case
$+(x_L)=0$ is completely analogous). Let 
\begin{equation}
 Q_{\eps,\lambda}^L=\left\{S:
b_\lambda^L\le\eps
\right\},
\label{Qepsilon}
\end{equation}
($b_\lambda^L$ was defined in Theorem \ref{th:Fquenchedcont}).
By Theorem \ref{th:Fquenchedcont} there exists $\widetilde L$ such that
for all $L\ge\widetilde L$ we have  $\mathbf P(Q_{\eps,\lambda}^L)>1-\eps$.\\
Then since both $F^L_S(x_L,{\lambda}/{(2L)^{d}})$
and $(\Gp_{\mathcal{B}}(x_L)+{1}/\lambda)/(\Gp_{\mathcal{B}}(0)+{1}/\lambda)$
are in $[0,1]$, for all $L\ge\widetilde L$
\[
   \sum_S \mathbf P(S)\left|F^L_S\left(x_L,\frac{\lambda}{(2L)^{d}}\right)-
\frac{\Gp_{\mathcal{B}}(x_L)+\frac{1}{\lambda}}{\Gp_{\mathcal{B}}(0)+\frac{1}{\lambda}}
	\right|\le 2\mathbf P((Q_{\eps,\lambda}^L)^c)+\eps\le 3\eps.
\]
\end{proof}
\begin{rem}
Clearly for all $\lambda>0$, if 
$$
\frac{\Gp_{\mathcal{B}}(x_L)+\frac{1}{\lambda}-\mathbf 1_{\{0\}}(x_L)}{\Gp_{\mathcal{B}}(0)
+\frac{1}{\lambda}}
$$ 
has a limit $f(\lambda)$ then by Theorem~\ref{th:Fanncont} we have 
that $F^L(x_L,{\lambda}/{(2L)^{d}})$ 
has limit $f(\lambda)$. 
\label{rem:limit}
\end{rem}
\begin{rem}
In discrete time one can show the same results with $2t_L$ instead of $t^2_L$ and constant
$\widetilde{\G}^{ev}_{\mathcal{B}}(x_L)$. 
As seen in the proof of Lemma~\ref{th:H1cont} the key of the proof in discrete time is that
two random walkers cannot meet before a time smaller than half of their initial distance
(while this is possible in continuous time, though it is unlikely that particles at initial
distance $t_L^2$ meet before time $t_L$). 
\label{rem_discreteG}
\end{rem}

\section{Meeting and hitting time of random walks}
\label{hit_time}


It is clear that, if $\Gp_{\mathcal{B}}(x_L)$
has a limit as $L$ goes to infinity, then Theorems \ref{th:Fquenchedcont}
and \ref{th:Fanncont} provide the limits of
$F^L_S(x_L,{\lambda}/{(2L)^{d}})$ and $F^L(x_L,{\lambda}/{(2L)^{d}})$.
The limit of $\Gp_{\mathcal{B}}(x_L)$ exists for instance in two particular
cases: $x_L=x$ for all $L$ sufficiently large, or $|x_L|\to \infty$.
In the first case clearly $\lim_{L\to\infty}\Gp_{\mathcal{B}}(x_L)=\Gp_{\mathcal{B}}(x)$.
In the second case, $\Gp_{\mathcal{B}}(x_L)$ converges to $0$ by the Dominated Convergence Theorem. 

\noindent \emph{Proof of Theorem \ref{th:main}.}
We prove the claim in continuous time. The proof in discrete time works in a similar way.
 \begin{enumerate}
\item By Theorem \ref{th:Fanncont} we know that for all $\lambda>0$
\begin{equation}\label{eq:limFL}
  F^L\left(x_L,\frac{\lambda}{(2L)^{d}}\right)\stackrel{L\to\infty}{\to}
\frac{\lambda \Gp_{\mathcal{B}}(x)+1-\lambda\mathbf 1_{\{0\}}(+(x))}{\lambda \Gp_{\mathcal{B}}(0)
+1}.
\end{equation}
Since for each $L$, $F^L$ is a monotone function of $\lambda$ and so is
the right hand side of \eqref{eq:limFL}, which is also continuous in $\lambda$, 
it follows that \eqref{eq:limFL} holds uniformly in $\lambda\ge0$.

Thus, if $x\neq0$, $T_L/(2L)^d$ converges in law (with respect to $\Prob^{x,0}$) to
\[
\frac{\Gp_{\mathcal{B}}(x)}{\Gp_{\mathcal{B}}(0)}\delta_0+
\left(1-\frac{\Gp_{\mathcal{B}}(x)}{\Gp_{\mathcal{B}}(0)}\right)
\exp\left(\frac 1{\Gp_{\mathcal{B}}(0)}\right),\]
while if $x=0$ then it converges to
\[
\left(1-\frac{1}{\Gp_{\mathcal{B}}(0)}\right)\delta_0+
\frac{1}{\Gp_{\mathcal{B}}(0)}
\exp\left(\frac 1{\Gp_{\mathcal{B}}(0)}\right).\]
Then \eqref{eq:T1} 
holds, and by monotonicity it holds 
uniformly in $t\ge0$.
\item
It follows as in the previous step using the fact that
$\Gp_{\mathcal{B}}(x_L)\to0$ uniformly in $\{x_L\}_L$ such that $|x_L|\ge\alpha_L$.
Indeed $\Gp_{\mathcal{B}}(x_L)=\int_0^{\infty}\Prob^0_{\mathcal{B}}(X_{2t}=+(x_L))dt$ goes to 0 by the Dominated Convergence Theorem since
$\Prob^0_{\mathcal{B}}(X_{2t}=+(x_L))\le \Prob^0_{\mathcal{B}}(X_{2t}=0)$
and $\int_0^{\infty}\Prob^0_{\mathcal{B}}(X_{2t}=0)dt\leq \G_{\mathcal{B}}(0)<\infty$.

\item
By Theorem~\ref{th:Fquenchedcont} we know that for all $n$
there exists $L_n$ such that 
$\mathbf P(S: b_\lambda^L (S) \le 1/n, \forall \lambda\in [1/n,n])\ge 1-1/n$ for all $L\ge L_n$.
Clearly the sequence $\{L_n\}_{n\ge1}$ is nondecreasing and for any $L\in [L_{n},L_{n+1})$
we may define 
\[
 H^L:=(S: b_\lambda^L (S) \le 1/n, \forall \lambda\in [1/n,n]).
\]
By Theorem~\ref{th:Fquenchedcont}  we have that $\mathbf{P}(H^L)\stackrel{L \to \infty}{\to} 1$.
If for all $L$ we choose $S\in H^L$ then for all $\lambda>0$
\begin{equation*}
F^L_S\left(x_L,\frac{\lambda}{(2L)^{d}}\right)\stackrel{L\to\infty}{\to}
\frac{\lambda \Gp_{\mathcal{B}}(x)+1}{\lambda \Gp_{\mathcal{B}}(0)+1}.
\end{equation*}
This, by an argument as in step $1$, proves that
\[
\Prob ^{x_{L},0}_{S}\left(\frac{T_{L}}{(2L)^{d}}>t\right)\stackrel{L\to \infty}{\to}
\left(1-\frac{\Gp_{\mathcal{B}}(x)}{\Gp_{\mathcal{B}}(0)}\right)\exp\left(-\frac{t}{\Gp_{\mathcal{B}}(0)}\right),\]
uniformly in $t\geq 0$. Since this convergence holds whenever we choose
for all $L$, $S\in H^L$, the assertion follows.


\item Choosing $S\in H^L$ as in previous step, uniformly with respect to $\{x_{L}\}_L$ such that either $|x_{L}|\geq \alpha_{L}$ or $d_{S}(0,x_{L})\geq \alpha_{L}$ we get
\[
\Prob ^{x_{L},0}_{S}\left(\frac{T_{L}}{(2L)^{d}}>t\right)\stackrel{L\to \infty}{\to}\exp\left(-\frac{t}{\Gp_{\mathcal{B}}(0)}\right),\]
uniformly in $t\geq 0$. This proves the claim.
\end{enumerate}
\hfill $\square$

\begin{rem}
Theorem \ref{th:main}.4 holds if we fix $0\in \Lambda(L)$ and we consider the supremum over all possible 
$x_{L}\in \Lambda(L)$ such that $d_{S}(x_{L},0)\geq \alpha_{L}$. 
We can repeat the same proof to show that the result still holds if we take 
the supremum over all possible pairs $(x_{L},y_{L})\in \Lambda(L)\times\Lambda(L)$ 
such that $d_{S}(x_{L},y_{L})\geq \alpha_{L}$. Namely, let $\alpha_L>t_L^2$ 
then for all $\eps>0$
\begin{equation}
  \mathbf P\left(S
: \sup_{(x_{L},y_{L})\in \Lambda(L)^{2}:d_{S}(x_{L},y_{L})\geq \alpha_{L}}
\left|\Prob_S^{x_L,y_{L}}\left(\frac{T_L}{(2L)^d}>t\right)\nonumber\\
-g(t)\right|
<\eps,\forall t\ge0\right)\stackrel{L\to\infty}{\to}1.
\label{eq:qlim_gen}
\end{equation}
\label{rem_unif}
\end{rem}

We observe that the same technique we employed to determine the
asymptotic behaviour of the first encounter time of two random walkers,
one starting at $x_L$ and the other at 0, may be used to obtain similar
results for the first time that a single random walker starting at $x_L$
hits 0. 

\begin{theorem}
Let $W_L$ be the first time that a random walk starting at $x_L$
hits 0 either in discrete or in continuous time. Then Theorem \ref{th:main} still holds with constant  $\G_{\mathcal{B}}(x)$ instead of $\Gp_{\mathcal{B}}(x)$.
\label{th:1walk}
\end{theorem}
\begin{proof}
\textit{(Discrete time)} The proof is analogous to the one of Theorem \ref{th:main} but easier, since we consider the return time of one single walk. Notice that the constant is the expected number of visits to $0$ of the discrete time random walk on the big world starting at $0$.\\  
\textit{(Continuous time)}
A standard approach (for instance use Slutsky theorems) allows to get the result starting from the one in discrete time. 
\end{proof}
\begin{rem}
As a corollary of Theorems \ref{th:main} and \ref{th:1walk} one can get a similar convergence 
result for random walkers starting from the stationary distribution $\pi$. The key is that the 
initial distance between the random walk and the origin 
(respectively between two random walks) is larger than $t^2_L$ with probability which converges 
to $1$ as $L$ goes to infinity, so that we are under hypothesis of 
Theorem \ref{th:main} either $2)$, in the annealed case, or $4)$ in the quenched one.   
\label{cor1walk}
\end{rem}

\section{Coalescing random walk on small world}
\label{coalescing}

The goal of this section is to prove a convergence result for coalescing random walk of $n$ particles 
on the small world. From now on we work on the continuous time process.

Let $\mathcal{I}(n)=\{\{x_{1},\ldots,x_{n}\}: x_{i} \in \Lambda(L), x_{i}\neq x_{j}\}$. 
Given $A \in \mathcal{I}(n)$, let $\{(X^{S}_{t}(x_{i}))_{t \geq 0}\}_{x_i\in A}$ be a family of
 independent random walks on small world $S \in \mathcal{S}^{L}$ such that $X^{S}_{0}(x_{i})=x_{i}$
and transition ruled by $P_S$ (recall Definition~\ref{def:law}). In the sequel we will drop the
superscript $S$ and simply write $X_t(x_i)$. 
We define for each $(x_{i},x_{j})\in \Lambda(L)\times \Lambda(L)$ and $S\in \mathcal{S}^{L}$
\begin{align*}
\tau(i,j):=&\inf \{s>0:X^{S}_{s}(x_{i})=X^{S}_{s}(x_{j})\}
\end{align*}
and for each $A\in \mathcal{I}(n)$
\begin{align*}
\tau(A):=&\inf_{\{x_i,x_j\}\subseteq A}\{\tau(i,j)\}.
\end{align*}
Let $\{\xi^{S}_{t}(A)\}_{t \geq 0}$ be the coalescing random walk starting from 
$A\in \mathcal{I}(n)$ on $S\in \mathcal{S}^{L}$, that is the process of $n$ independent 
random walks subjected to the rule that when two particles reach the same site they coalesce 
to one particle. 
Given a probability measure $\mu$ on $\Lambda(L)^{n}$, we denote by $\PP^{\mu}_{S}$ the law of 
the coalescing random walk on $S$ with initial probability $\mu$ and transitions ruled by $P_{S}$. 
If $\mu=\delta_{A}$ with $|A|=n$, we write $\PP^{A}_{S}$.

Let $|\xi^{S}_{t}(A)|$ be the number of particles of $\xi^{S}_{t}(A)$ at time $t$. 
When not necessary we omit the dependence on $S$ and we simply write $\{\xi_t(A)\}_{t\geq 0}$, 
$X_{t}(x_{i})$, $\tau(i,j)$ and $\tau(A)$.\\
The Kingman's coalescent is a Markov process $(D_{t})_{t\geq 0}$ on $\{0,1,\ldots,n\}$ with transition mechanism
$$
n\to n-1 \text{ at rate }\binom{n}{2}.
$$
The law $P_{n}(D_{t}=k)=q_{n,k}(t)$ is given by
\begin{align*}
q_{n,k}(t)=&\sum_{j=k}^{n}\frac{(-1)^{j+k}(2j-1)(j+k-2)!\binom{n}{j}}{k!(k-1)!(j-k)!\binom{n+j-1}{j}}\exp\left(-t\binom{j}{2}\right);\\
q_{\infty,k}(t)=&\sum_{j=k}^{\infty}\frac{(-1)^{j+k}(2j-1)(j+k-2)!}{k!(k-1)!(j-k)!}\exp\left(-t\binom{j}{2}\right).
\end{align*}
see for instance \cite{cf:Cox}, \cite{cf:Tavare}.\\ 
We define 
\begin{align}
\mathcal{A}^{L}(h,n):=&\left\{A\in \mathcal{I}_{n}:d(x_{i},x_{j}) > h, \text{ for all } i\neq j\right\}
\label{ALn}\\
\mathcal{A}_{S}^{L}(h,n):=&\left\{A\in \mathcal{I}_{n}: d_{S}(x_{i},x_{j})> h,  \text{ for all } i \neq j\right\}
\label{ASLn}\
\end{align}
the set of $n$-uples with distance larger than $h$ respectively on $\Lambda(L)$ and on a fixed small world $S$. 
Notice that $\mathcal{A}^{L}_{S}(h,n)\subseteq \mathcal{A}^{L}(h,n)$ for all $S \in \mathcal{S}^L$. 
Given $A\in \mathcal{A}^{L}(h,n)$, we introduce
\begin{align}
\mathcal{D}(A):=&\left\{S \in \mathcal{E}^L: A \in \mathcal{A}^{L}(h,n)\setminus \mathcal{A}_{S}^{L}(h,n)\right\}.
\label{D}
\end{align} 
Remember that we focus on the nearest neighbour case, but all results can be extended to the case with neihbourhood structure given by $\mathcal{N}^{\infty}_{m}$.\\

We begin from $n$ particles in $A\in \mathcal{A}^{L}(h,n)$. We prove that by taking a particular $h:=h_{L}$ and $L$ large we get that $A \in \mathcal{A}_{S}^{L}(h,n)$ with large probability. We assume
\begin{align}
i)\quad h_{L}\geq t^2_{L} \qquad ii)\lim_{L \to \infty}\frac{M^{4h_{L}}}{(2L)^{d}}=0 
\label{ip}
\end{align}
where $M=(2m+1)^{d}$ or $M=2d+1$ depending on the neighbourhood structure we work with. Note that hypothesis \eqref{ip} are satisfied if $h_{L}=t^2_{L}$. 
\begin{lem}
If \eqref{ip} holds, for each $n<\infty$, $\eps>0$ there exists $\widetilde{L}$ such that for each $L>\widetilde{L}$ and $A\in \mathcal{A}^{L}(h_{L},n)$,
\begin{equation*}
\mathbf P\left(\mathcal{D}(A)\right)<\eps.
\end{equation*}
\label{lemD}
\end{lem}
\begin{proof}
Let $A \in \mathcal{A}^{L}(h_{L},n)$. If $S\in \mathcal{D}(A)$, then there exists at least one pair of elements $(x_{i},x_{j})\in A\times A$, $i \neq j$, such that $d_{S}(x_{i},x_{j})<h_{L}$. 
By \eqref{D} and \eqref{eq:stima>02} 
\begin{align*}
\mathbf P\left(\mathcal{D}(A)\right)=& \mathbf P\left(S\in \mathcal{E}^{L}: \exists (x_{i},x_{j})\in A\times A: d_{S}(x_{i},x_{j})
\leq h_{L}\right)\leq  n^{2}
\frac{CM^{4h_{L}}}{L^{d}}.
\end{align*}
Since $n$ is fixed, the claim follows by \eqref{ip} $(ii)$.
\end{proof}
Therefore given $A \in \mathcal{A}^{L}(h_{L},n)$ with large probability $A\in \mathcal{A}^{L}_{S}(h_{L},n)$.\\
By Remark \ref{rem_unif}, if $\alpha_{L}\geq t_{L}^2$, there exists a sequence $\{\widetilde{H}^L\}_{L}$ with $\widetilde{H}^L\subseteq \mathcal{S}^{L}$ such that $\mathbf P(\widetilde{H}_L)\stackrel{L\to \infty}{\to}1$ and for each sequence $\{S^L\}_{L}$ with $S^L \in \widetilde{H}^L$  
\begin{equation}
\sup_{(x_{L},y_{L}):d_{S}(x_{L},y_{L})\geq \alpha_{L}}\left|\Prob_S^{x_L,y_{L}}\left(\frac{T_L}{(2L)^d}>t\right)
-\exp\left(-\frac {t}{\Gp_{\mathcal{B}}(0)}\right)\right|\stackrel{L \to \infty}{\to}0
\label{seq_H}
\end{equation}
Note that \eqref{seq_H} still holds for the sequence $\{Q^L\cap \widetilde{H}^L\}_{L}$ and $\mathbf P(Q^L\cap \widetilde{H}_L)\stackrel{L\to \infty}{\to}1$. Let $H^L:=\widetilde{H}^L\cap Q^L$.\\

The following lemma states that, starting from $4$ particles in a set of small world with large probability, when two particles meet the others are distant.
\begin{lem}
Assume \eqref{ip}. For each $\eps>0$ there exists $\widetilde{L}$ such that for each $L>\widetilde{L}$, $S \in H^L$ and $A \in \mathcal{A}^{L}_{S}(h_{L},4)$,
\begin{align}
&\int_{0}^{\infty}\mathbb{P}^{A}_{S}\left(\tau(1,2)\in ds, d_{S}(X_{s}(x_{1}),X_{s}(x_{3}))\leq h_{L}\right)<\eps, \label{inc1}\\
\nonumber\\
&\int_{0}^{\infty}\mathbb{P}^{A}_{S}\left(\tau(1,2)\in ds, d_{S}(X_{s}(x_{3}),X_{s}(x_{4}))\leq h_{L}\right)<\eps. \label{inc2}
\end{align}
\label{inc}
\end{lem}
\begin{proof}
We prove \eqref{inc1}; \eqref{inc2} can be proved in a similar way.  We split the integral in two 
parts. By Theorem \ref{th:main}.4 and by \eqref{ip} $(ii)$, for each $\eps>0$ there exists 
$\widetilde{L}$ such that for each $L>\widetilde{L}$
\begin{align}
\int_{0}^{\frac{d}{\gamma}\log (2L)}&\mathbb{P}^{A}_{S}\left(\tau(1,2)\in ds, d_{S}(X_{s}(x_{1}),X_{s}(x_{3}))\leq h_{L}\right)
\leq\int_{0}^{\frac{d}{\gamma}\log (2L)}\mathbb{P}^{A}_{S}\left(\tau(1,2)\in ds \right)\nonumber \\
&=1-\exp\left({-\frac{d \log (2L)}{\gamma \Gp_{\mathcal{B}}(0)(2L)^{d}}}\right)+\eps/6 <\eps/3,
\label{I0}
\end{align}
where $\gamma$ is given by \eqref{eq:estim_cont1} and does not depend on $S$ since we are choosing
$S\in Q^L$.
The second part is 
\begin{align*}
&\int_{\frac{d}{\gamma}\log (2L)}^{\infty}
\mathbb{P}^{A}_{S}\left(\tau(1,2)\in ds, d_{S}(X_{s}(x_{1}),X_{s}(x_{3}))\leq h_{L}\right)\\
&\leq \int_{\frac{d}{\gamma}\log (2L)}^{\infty}\sum_{y\in \Lambda(L)}
\mathbb{P}^{A}_{S}\left(\tau(1,2)\in ds, X_{s}(x_{1})=y\right)\sum_{z:d_{S}(y,z)\leq h_{L}}
\left|\mathbb{P}^{x_{3}}_{S}(X_{s}=z)-\frac{1}{(2L)^{d}}\right|\\
&+\int_{\frac{d}{\gamma}\log (2L)}^{\infty}\sum_{y\in \Lambda(L)}
\mathbb{P}^{A}_{S}\left(\tau(1,2)\in ds, X_{s}(x_{1})=y\right)\sum_{z:d_{S}(y,z)\leq h_{L}}\frac{1}{(2L)^{d}}:=I(1)+I(2).
\end{align*}
Since the number of sites $z$ such that $d_{S}(y,z)\leq h_{L}$ is at most $M^{h_{L}}$ for each $y \in \Lambda(L)$, for each $L$ large enough we get
\begin{align}
I(2) \leq & 
\int_{\frac{d}{\gamma}\log (2L)}^{\infty}\mathbb{P}^{A}_{S}\left(\tau(1,2)\in ds\right)\frac{M^{h_{L}}}{(2L)^{d}}
 = \mathbb{P}^{x_{1},x_{2}}_{S}\left(T_{L}> \frac{d}{\gamma}\log (2L)\right)\frac{M^{h_{L}}}{(2L)^{d}}\leq \eps/3
\label{I2}
\end{align}
 by \eqref{ip} $(ii)$. Note that if $s \geq \frac{d}{\gamma}\log (2L)$ then $e^{-\gamma s} \leq \frac{1}{(2L)^{d}}$;
 therefore  
by \eqref{eq:estim_cont1} then $I(1)$ is smaller or equal to
\begin{align}
\int_{\frac{d}{\gamma}\log (2L)}^{\infty}\sum_{y\in \Lambda(L)}&\mathbb{P}^{A}_{S}\left(\tau(1,2)\in ds, X_{s}(x_{1})=y\right)\sum_{z:d_{S}(y,z)\leq h_{L}}e^{-\gamma s}\nonumber\\
\leq & \int_{\frac{d}{\gamma}\log (2L)}^{\infty}\sum_{y\in \Lambda(L)}\mathbb{P}^{A}_{S}\left(\tau(1,2)\in ds, X_{s}(x_{1})=y\right)\frac{M^{h_{L}}}{(2L)^{d}}< \eps/3
\label{I12}
\end{align}
and the claim follows by \eqref{I0}, \eqref{I2} and \eqref{I12}.
\end{proof}
\begin{rem}
Since $S \in H^L$, Lemma \ref{inc} still holds if for all $A \in \mathcal{A}^{L}(h_{L},n)$ we choose $S \in H^L\cap \mathcal{D}(A)^{c}$. Moreover by 
Lemma \ref{lemD} and \eqref{seq_H} such a set has probability which converges to $1$ as $L$ goes to infinity.
\label{remset}
\end{rem}
We prove that the number of particles in the rescaled coalescing random walk converges in law to the number of particles of a Kingman's coalescent.  A similar approach has been used for 
\cite[Theorem $5$]{cf:Cox} and in \cite{cf:CoxGriffeath}.\\
We work by induction on the number of particles $n$. If $n=2$, the induction basis is given by Theorem \ref{th:main}.4. The following lemma shows that the assertion is true before the first collision of two particles.
\begin{lem}
Assume \eqref{ip}.
For each $n\in \mathbb{N}$, $T>0$, $A \in \mathcal{A}^{L}(h_{L},n)$, and 
$\eps>0$ there exists $\widetilde{L}$ such that for each $L>\widetilde{L}$,  $S \in H^L\cap \mathcal{D}(A)^{c}$ and $0\leq t\leq T$, 
$$
\left|\PP_{S}\left(|\xi_{s_{L}t}(A)|=n\right)-\exp\left(-\binom{n}{2}t\right)\right|<\eps
$$
where $s_{L}:=(2L)^{d}\Gp_{\mathcal{B}}(0)$.
\label{coalescenten}
\end{lem}
\begin{proof}
Note that $\PP_{S}(|\xi_{s_{L}t}(A)|=n)$ and $ \exp\left(-\binom{n}{2}t\right)$ are non-increasing monotone $t$ functions.  We define, for each pair $\{i,j\}\subseteq \{1,2,\ldots, n\}$,
\begin{align*}
&H_{t}(i,j):=\{\tau(i,j) \leq s_{L}t\}; \qquad q_{t}=q_{t}(A):=\mathbb{P}(\tau(A) \leq s_{L}t).
\end{align*}
For all $S \in H^L \cap \mathcal{D}(A)^{c}$,
\begin{equation}
\PP^{A}_{S}\left(H_{t}(i,j)\right)=\PP^{A}_{S}\left(\tau=\tau(i,j)\leq s_{L}t \right)+\sum_{\{k,l\}\neq \{i,j\}}\int_{0}^{s_{L}t}\PP^{A}_{S}\left(\tau=\tau(k,l)\in ds, \tau(i,j)\leq s_{L}t\right)
\label{tempeq}
\end{equation}
Each term of the sum on the right hand side is equal to
\begin{align*}
\int_{0}^{s_{L}t}\sum_{y,z}\PP^{A}_{S}&\left(\tau=\tau(k,l)\in ds,X_{s}(x_{i})=y,X_{s}(x_{j})=z,\tau(i,j)\leq s_{L}t\right).
\end{align*}
By Lemma \ref{inc} for all $L$ sufficiently large
\begin{multline}\label{eq:int1}
\displaystyle \int_{0}^{s_{L}t}\sum_{y}\sum_{z:d_{S}(y,z)\leq h_{L}}\PP^{A}_{S}\left(\tau=\tau(k,l)
\in ds,X_{s}(x_{i})=y,X_{s}(x_{j})=z,\tau(i,j)\leq s_{L}t\right)\\
\displaystyle \leq \int_{0}^{\infty}\PP^{A}_{S}\left(\tau=\tau(k,l)\in ds,d_{S}(X_{s}(x_{i}),
X_{s}(x_{j}))\leq h_{L}\right)\leq \eps/(8n^4)
\end{multline}
for all choices of $S \in H^L\cap \mathcal{D}(A)^{c}$, $\{i,j\}\subseteq \{1,\ldots n\}$ and $t\geq 0$.
We are left with evaluating
\begin{align} \label{eq:int2}
\int_{0}^{s_{L}t}\sum_{y}\sum_{z:d_{S}(y,z)> h_{L}}\PP^{A}_{S}\left(\tau=\tau(k,l)\in ds,X_{s}(x_{i})=y,X_{s}(x_{j})=z\right)\PP^{y,z}_{S}
\left(T_{L} \leq s_{L}t-s\right).
\end{align}
By 
Theorem~\ref{th:main}.2, $|\Prob_{S}^{y,z}(T_{L}\leq s_{L}t-s)-1+\exp(-t+s/s_L)|<\eps/(8n^4) $ 
for all $L$ sufficiently large and for all choices of $S \in H^L\cap \mathcal{D}(A)^{c}$, $y$ 
and $z$ such that $d_{S}(y,z)\geq h_{L}$, $0\leq s \leq t$. 
Then equation \eqref{eq:int2} does not differ by more than $\eps/(4n^4)$ from  
\begin{equation}\label{eq:tempco}
\int_{0}^{s_{L}t}\PP^{A}_{S}(\tau=\tau(k,l)\in ds)\left(1-\exp(-t+s/s_{L})\right).
\end{equation}
Indeed the difference between \eqref{eq:int2} and \eqref{eq:tempco} is not larger than the sum
of \eqref{eq:int1} and
\[
 \int_0^\infty |\Prob_{S}^{y,z}(T_{L}\leq s_{L}t-s)-1+\exp(-t+s/s_L)|\PP^{A}_{S}(\tau=\tau(k,l)\in ds),
\]
which is not larger than $\eps/(8n^4)$ if $L$ is sufficiently large.
Integrating by parts and changing variables, we get
\begin{align}
&\int_{0}^{s_{L}t}\PP^{A}_{S}\left(\tau=\tau(k,l)\in ds\right)\left(1-\exp(-t+s/s_{L})\right) \nonumber\\
&=\int_{0}^{s_{L}t}\PP^{A}_{S}\left(\tau=\tau(k,l)\leq s\right)\frac{1}{s_{L}}\exp\left(-t+s/s_{L}\right)ds \nonumber\\
&=\int_{0}^{t}\PP^{A}_{S}\left(\tau=\tau(k,l)\leq s_{L}u\right)\exp\left(-(t-u)\right)du.
\label{tempmin2}
\end{align}
By Theorem~\ref{th:main}.2, for all $L$ sufficiently large $|\PP^{A}_{S}(H_{t}(i,j)\leq t)-
(1-e^{-t})|\leq \eps/(4n^2)$ 
for all $S \in H^L\cap \mathcal{D}(A)^{c}$, $(i,j)\subseteq \{1,\ldots n\}$ and $t\geq 0$. 
Summing over all pairs of $i$ and $j$ on 
\eqref{tempeq} and using \eqref{tempmin2}
\begin{align*}
q_t=&\sum_{\{i,j\}}\PP^{A}_{S}\left(\tau=\tau(i,j)\leq s_{L}t \right)\\
=&\sum_{i,j}\PP^{A}_{S}\left(H_{t}(i,j)\right)-\sum_{\{i,j\}}\sum_{\{k,l\}\neq \{i,j\}}\int_{0}^{s_{L}t}
\PP^{A}_{S}\left(\tau=\tau(k,l)\in ds,\tau(i,j)\leq s_{L}t\right)\\
=& \binom{n}{2}(1-e^{-t})-\left(\binom{n}{2}-1\right)e^{-t}\int_{0}^{t}q_se^sds+R
\end{align*}
where the modulus of $R$, for all $L$ sufficiently large for all choices of $S \in H^L\cap \mathcal{D}(A)^{c}$, $y$ and $z$ such that $d_{S}(y,z)\geq h_{L}$ and for all $0\leq t\leq T$ is smaller than $\eps/2$. 
We know (see \cite[Lemma 2]{cf:CoxGriffeath}) that if 
$$
u^{L}(t)=\binom{n}{2}(1-e^{-t})-\left(\binom{n}{2}-1\right)e^{-t}\int_{0}^{t}u^{L}(s)e^sds+R
$$
then for $L$ large enough $u^{L}(t)$ does not differ by more than $\eps/2$ from $u(t)$, the solution of  
$$
u(t)=\binom{n}{2}(1-e^{-t})-\left(\binom{n}{2}-1\right)e^{-t}\int_{0}^{t}u(s)e^sds
$$
which is
$$
u(t)=1-\exp\left(-\binom{n}{2}t\right)
$$
and the claim follows.
\end{proof}
We are now ready to prove the final result.\\

\noindent \emph{Proof of Theorem \ref{npart}.}
We fix $A \in \mathcal{A}^L(h_L,n)$ and we show \eqref{eq:coalescing} by induction on $n$. 
Theorem \ref{th:main} gives the result when $n=2$ for all $k$ (that is $k=2$) and 
Lemma \ref{coalescenten} gives the result for $n$ and $k=n$.\\ 
Suppose the result holds for $n-1$ for all $k$. We have to prove it for $n$ and $k<n$.
\begin{align}
\PP^{A}_{S}(|\xi_{s_{L}t}(A)|<k)&=\int_{0}^{s_{L}t}\PP^{A}_{S}(\tau \in ds,|\xi_{s_{L}t}(A)|<k)\nonumber\\
&=\int_{0}^{s_{L}t}\sum_{B\in I(n-1)}\PP^{A}_{S}(\tau \in ds,\xi_{s}(A)=B)\PP^{B}_{S}(|\xi_{s_{L}t-s}(B)|<k).
\label{tmp_ncoal}
\end{align}
Using Lemma \ref{inc}, if $B \notin \mathcal{A}_S^L(h_L,n-1)$, for all $L$ sufficiently large 
\begin{align*}
&\int_{0}^{s_{L}t}\sum_{B\notin \mathcal{A}_S^L(h_L,n-1)}\PP^{A}_{S}(\tau \in ds,\xi_{s}(A)=B)\PP^{B}_{S}(|\xi_{s_{L}t-s}(B)|<k)\\
& \leq \sum_{\{i,j\}}\sum_{\{k,l\}\neq \{i,j\}}\int_{0}^{s_{L}t}\PP^{A}_{S}(\tau(i,j) \in ds,d_{S}(X_{s}(x_{k}),X_{s}(x_{l})\leq h_{L})<\eps/3
\end{align*}
since $n$ is fixed, for each $S \in H^L\cap \mathcal{D}(A)^{c}$, $t\geq 0$.\\
Changing variables, setting $s=s_{L}v$, then \eqref{tmp_ncoal} is equal to 
\begin{align*}
&\int_{0}^{t}\sum_{B\in \mathcal{A}_S^L(h_L,n-1)}\PP^{A}_{S}(\tau \in s_{L}dv,\xi_{s_{L}v}(A)=B)\PP^{B}_{S}(|\xi_{s_{L}(t-v)}(B)|<k)+R.
\end{align*}
where the modulus of $R$ is smaller than $\eps/3$ for all $L$ sufficiently large for all choices of $A\in \mathcal{A}^L(h_L,n)$, $S \in H^L\cap \mathcal{D}(A)^{c}$, $0\leq t\leq T$. By induction hypothesis, for all $L$ sufficiently large 
$$
\left|\PP^{B}_{S}(|\xi_{s_{L}(t-s)}(B)|<k)-P_{n-1}(D_{t-s}<k)\right|<\eps/3
$$
for $B \in \mathcal{A}_S^L(h_L,n-1)$ and for each $S\in H^L\cap \mathcal{D}(A)^{c}$ and $0 \leq s \leq t$. Thus the last term of the previous integral differs at most by $\eps$ from
$$
\int_{0}^{t}\PP^{A}_{S}\left(\frac{\tau}{s_{L}}\in dv\right)P_{n-1}(D_{t-v}<k)=-\int_{0}^{t}\PP^{A}_{S}\left(\frac{\tau}{s_{L}}\leq v\right)\frac{d}{dv}P_{n-1}(D_{t-v}<k)dv
$$
after an integration by parts.
Note that $v \to P_{n-1}(D_{t-v}=k)$ is a continuous function; therefore by definition of Kingman's coalescent and because the right hand side $P_{n}(D_{t}<k)$ is finite, we get 
(see \cite{cf:Cox}) 
\begin{align*}
\PP^{A}_{S}(|\xi_{s_{L}t}(A)|<k)=& \sum_{i=1}^{k-1}\int_{0}^{t}\binom{n}{2}\exp\left(-\binom{n}{2}v\right)P_{n-1}(D_{t-v}=k)dv+R\\
= & \sum_{i=1}^{k-1}P_{n}(D_{t}=k)+R=P_{n}(D_{t}<k)+R
\end{align*}
where the modulus of $R$, for all $L$ sufficiently large, 
for all choices of $S\in H^L\cap \mathcal{D}(A)^{c}$ and $0\leq t\leq T$ is smaller than $\eps$.
\hfill $\square$
\begin{rem}
In Theorem \ref{npart} we fix $A \in \mathcal{A}^{L}(h_L,n)$ and the result holds in a sequence of small world graphs depending on $A$. One can prove that the same result holds for the sequence $(H^L)_L$ uniformly in $\mathcal{A}_S^{L}(h_L,n)$ and $S\in H^L$. 
\label{npart_unif}
\end{rem}
\begin{rem}
By summing over all realizations of the small world graph, one can get the annealed result as a corollary of 
Theorem \ref{npart}.
\end{rem}

\noindent \textbf{Aknowledgements.}
D.Borrello thanks Laboratoire MAP5, Universit\'e Paris Descartes for hospitality and 
the support of the French
Ministry of Education through the ANR BLAN07-218426 grant. The authors wish to thank
Elisabetta Candellero for fruitful discussions on the comparison of
$G_{\Z^d}(0)$ and $G_{\mathcal B}(0)$.

\renewcommand{\appendixpagename}{Appendix}
\appendixpage

\appendix 
\section{Comparison with the d-dimensional torus}
\label{App}
As observed in the introduction, if we consider the usual neighbourhood structure on $\Z^d$,
then the big world is the Cayley graph of $\Z^d*\Z_2$. 
Recall that we are given a transition matrix $\Delta$ which defines a random walk
on $\Z^d$ and a positive number $\beta$ which gives the probability of moving in the
$\Z_2$ direction on the big world (one moves with probability $1-\beta$ in the $\Z^d$
component).
In order to compare the asymptotic
behaviour of the meeting time of two walkers on the small world and on the torus,
we need to compare $\Gp_{\mathbb{Z}^{d}}(0)$ with $\Gp_{\mathcal{B}}(0)$. 
Proposition \ref{beta} gives some information in this direction.
\begin{pro}
Let $d\geq 3$, $\Delta$ be the transition matrix of an adapted,
 translation invariant symmetric random walk
on $\Z^d$ and $\beta>0$.\\
$i)$ There exists $\beta_{1}>0$ such that $\Gp_{\mathbb{Z}^{d}}(0)< \Gp_{\mathcal{B}}(0)$ for each $\beta\in [\beta_{1},1]$.\\
$ii)$ There exists $\beta_{2}>0$ such that  $\Gp_{\mathbb{Z}^{d}}(0)> \Gp_{\mathcal{B}}(0)$ for each $\beta\in (0,\beta_{2}]$.
\label{beta}
\end{pro}
\begin{proof}
Since $\Gp_{\mathcal{B}}(0)=\G_{\mathcal{B}}(0)/2$ (we defined $\G_{\mathcal{B}}(0)$ in 
Section \ref{bigworld}) and $\Gp_{\mathbb{Z}^{d}}(0)=
\G_{\Z^d}(0)/2$ (where $\G_{\Z^d}(0)$ is the expected time spent at 0 by the walk on $\Zd$), we prove that $\G_{\mathbb{Z}^{d}}(0)$ is smaller (respectively larger) than $\G_{\mathcal{B}}(0)$ for $\beta$ large (respectively small) enough.\\
$i)$ Since $\Prob^0_{\mathcal{B}}(X_{2n}=0)\geq \beta^{2n}$ (note that one possible trajectory
of the walk is the one which from 0 takes the long range edge and back $n$ times) we get
\begin{align*}
\G_{\mathcal{B}}(0)\geq & \sum_{n=0}^{\infty}\beta^{2n}=\frac{1}{1-\beta^{2}}.
\end{align*}
and the claim follows by taking $\beta$ close to $1$ since $G_{\mathbb{Z}^{d}}(0)<\infty$ if $d\geq 3$.\\
$ii)$ Let  $\widehat G$ be the Green function of the Markov chain $Y$ on $\Z^d$
which has transition matrix $\Delta$, and
$\widehat F$ be the generating function of its first time returns
\begin{align*}
\widehat{G}(z)=\sum_{n=0}^{\infty}\Prob^0(Y_{n}=0)z^{n}; \qquad 
\widehat{F}(z)=\sum_{n=0}^{\infty}\Prob^{0}(Y_{n}=0,Y_k\neq 0 \text{ for all } k<n)z^{n}.
\end{align*}
Note that $\widehat G(1)=G_{\Z^d}(0)$.
By \cite[Proposition 9.10]{cf:Woess}, there exists $r>0$ and a function $\Phi(\cdot)$ such that
\begin{equation}
\widehat{G}(z)=\Phi(z\widehat{G}(z)), \qquad z \in [0,r).
\label{PhiG}
\end{equation}
Moreover $\Phi\in \mathcal C^2$ and it is strictly increasing and strictly convex.\\
Let $P$ be the transition matrix on the big world. 
 We denote by $\Phi_{\Z^d \ast \Z_{2}}$, $\Phi_{\Z^d}$ and $\Phi_{\Z_{2}}$ the functions which satisfy 
\eqref{PhiG} respectively for the Markov chains $X$ on the big world, $Y$ on $\Z^d$ and 
the simple random walk on $\Z_2$.
The function $\Phi_{\Z_2}(t)$ can be computed explicitly,
\[
 \Phi_{\Z_2}(t)=\frac12 (1+\sqrt{1+4t^2}).
\]
By \cite[Theorem 9.19]{cf:Woess}
\begin{equation}\label{eq:phi}
 \Phi_{\Z^d*\Z_2}(t)=\frac12 (1+\sqrt{1+4\beta^2t^2})+\Phi_{\Z^d}((1-\beta)t)-1.
\end{equation}
We denote by $\widehat{G}_\beta=\widehat{G}_{\Z^d \ast \Z_2}(1)=\G_{\mathcal{B}}(0)$, and
by $\widehat{G}=\widehat{G}_{\Z^d}(1)=G_{\Z^d}(0)$. Note that, by \eqref{PhiG}
$\widehat{G}_{\beta}$ is a fixed point of $\Phi_{\Z^d \ast \Z_2}$, while $\widehat G$
is a fixed point of $\Phi_{\Z^d}$. 
We write \eqref{eq:phi} with $t=\widehat{G}_\beta$:
\begin{equation}
 \label{eq:appendix}
 \widehat{G}_\beta=-\frac12 +\frac12\sqrt{1+4\beta^2\widehat{G}_\beta^2}+
\Phi_{\Z^d}((1-\beta)\widehat{G}_\beta).
\end{equation}
Our goal is to write the second member of \eqref{eq:appendix} in a neighbourhood of $\beta=0$,
as a function of $\widehat G$.

We note that $\displaystyle \lim_{\beta\to 0}\widehat{G}_\beta=\widehat{G}$: to prove this 
denote by $X$ and $Y$  the random walks on $\Z^d*\Z_2$ and on $\Z^d$
respectively, both starting from the identity 0 of the group.
Then, since every trajectory from 0 to 0 on $\Z^d*\Z_2$ projects onto a trajectory from
0 to 0 in $\Z^d$, we have
\[
 \Prob(X_{2n}=0)\le\sum_{k=0}^n\Prob(Y_{2n-2k}=0)\beta^{2k}.
\] 
Recognizing in the second term the general term of the product of two series, we get
\[
 \widehat G_\beta=\sum_{n=0}^\infty\Prob(X_{2n}=0)\le \sum_{n=0}^\infty\Prob(Y_{2n}=0)
\sum_{k=0}^\infty\beta^{2k}=\frac{\widehat G}{1-\beta^2},
\]
whence $\limsup_{\beta\to0}\widehat G_\beta\le\widehat G$.\\
Let $A_m$ be the event that the trajectory of $X$ up to step $2m$ lies entirely in the first copy of
$\Z^d$: for all $m$
\[
 \widehat G_\beta\ge\sum_{n=0}^m\Prob(X_{2n}=0)=\sum_{n=0}^m\Prob(X_{2n}=0|A_m)(1-\beta)^{2m}
+\sum_{n=0}^m\Prob(X_{2n}=0|A_m^c)(1-(1-\beta)^{2m}).
\]
Note that $\sum_{n=0}^m\Prob(X_{2n}=0|A_m^c)\le m$, thus for all $m$,
\[
 \liminf_{\beta\to0}\widehat G_\beta\ge \liminf_{\beta\to0}\sum_{n=0}^m\Prob(X_{2n}=0|A_m)=
\sum_{n=0}^m\Prob(Y_{2n}=0),
\]
and $\liminf_{\beta\to0}\widehat G_\beta\ge\widehat G$.

Notice that as $\beta\to 0$
\[
 \sqrt{1+4\beta^2\widehat{G}_\beta^2}=1+2\beta^2\widehat{G}_\beta^2+o(\beta^3\widehat{G}_\beta^3).
\]
and by Taylor expansion of $\Phi_{\Z^d}$ centered at $\widehat{G}$ with Lagrange form of the remainder: 
\[
 \Phi_{\Z^d}((1-\beta)\widehat{G}_\beta)=
\widehat{G}+\Phi_{\Z^d}^\prime(\widehat{G})\left[(1-\beta)\widehat{G}_\beta-\widehat{G}\right]+
\frac12\Phi_{\Z^d}^{''}(y)(y-\widehat{G})^2,
\]
where $y$ is between $\widehat{G}$ and $(1-\beta)\widehat{G}_\beta$. Two useful formulas 
for $\Phi^\prime$ and $\Phi^{''}$ can be found in \cite[p.99]{cf:Woess}:
\[
 \Phi^\prime(t)=1/(z+\widehat{G}(z)/\widehat{G}^\prime(z)),\qquad
\Phi^{''}(t)=(\widehat{G}(z)/(\widehat{G}(z)+z\widehat{G}^\prime(z)))^3\widehat{F}^{''}(z),
\]
where $z$ is such that $t=z\widehat{G}(z)$. If $t=\widehat{G}$ then 
$$
\Phi^\prime_{\Z^d}(\widehat{G})=\frac{\widehat{G}^\prime}{\widehat{G}^\prime+\widehat{G}}, \qquad
\Phi^{''}_{\Z^{d}}(\widehat{G})=\left(\frac{\widehat{G}}{\widehat{G}+\widehat{G}^\prime}\right)^3
\widehat{F}^{''}(1),
$$
where $\widehat{G}^\prime=\frac{d}{dz}\widehat{G}_{\Z^d}(z)|_{z=1^-}$. 
Therefore we may write \eqref{eq:appendix} as
$$
\widehat{G}_{\beta}-\widehat{G}=\beta^{2}\widehat{G}_{\beta}^{2}+
\frac{\widehat{G}^\prime}{\widehat{G}^\prime+
\widehat{G}}\left[\widehat{G}_{\beta}-\widehat{G}-\beta \widehat{G}_{\beta}\right]+
\frac12\Phi_{\Z^d}^{''}(y)(y-\widehat{G})^2+o(\beta^3 \widehat{G}^3_{\beta})
$$
Since $(y-\widehat{G})^2\leq (\widehat{G}_{\beta}-\widehat{G}-\beta \widehat{G}_\beta)^2$ we get
\begin{multline*}
 (\widehat{G}_{\beta}-\widehat{G})\frac{\widehat{G}}{\widehat{G}^\prime+\widehat{G}}\leq 
\beta^2 \widehat{G}^2_{\beta}-\beta \widehat{G}_\beta
\frac{\widehat{G}^\prime}{\widehat{G}^\prime+\widehat{G}}\\
+\frac12\Phi_{\Z^d}^{''}(y)
\left[(\widehat{G}_\beta-\widehat{G})^2+\beta^2\widehat{G}_\beta^2-2\beta 
\widehat{G}_\beta(\widehat{G}_\beta-\widehat{G})+o(\beta^3 \widehat{G}_{\beta}^3)\right].
\end{multline*}
Thus
\begin{multline*}
(\widehat{G}_{\beta}-\widehat{G})\left[\frac{\widehat{G}}{\widehat{G}^\prime+\widehat{G}}+
\Phi_{\Z^d}^{''}(y)\beta \widehat{G}_\beta-frac12\Phi_{\Z^d}^{''}(y)(\widehat{G}_\beta-
\widehat{G})\right]\\
\leq  
-\beta \widehat{G}_\beta\left[\frac{\widehat{G}^\prime}{\widehat{G}^\prime+\widehat{G}}+
\beta \widehat{G}_\beta(1+frac12\Phi_{\Z^d}^{''}(y)+o(\beta \widehat{G}_{\beta}))\right].
\end{multline*}
Note that by convexity $\Phi^{''}> 0$, and
by continuity 
we get that $\Phi^{''}_{\Z^d}(y)\stackrel{\beta\to 0}{\to}\Phi^{''}_{\Z^d}(\widehat G)>0$.
Then the coefficient of $(\widehat{G}_\beta-\widehat{G})$ on the left hand side 
is strictly positive when $\beta$ is small;  while
the coefficient of $\beta \widehat{G}_{\beta}$ on the right hand side 
is strictly negative when $\beta$ is small. Whence $\widehat G_\beta-\widehat G$
has to be negative for $\beta$ sufficiently small and the 
claim follows.
\end{proof}


\begin{thebibliography}{breitestes Label}

\bibitem{cf:AlbertBarabasi} {\sc R. Albert and A.L. Barab\'asi}. 
Statistical mechanics of complex networks. \emph{Rev. Modern Phys.}\textbf{74}, (2002), no. 1, 47--97.

\bibitem{cf:BarbourReinert} {\sc A.D. Barbour and G. Reinert}. Small worlds. \emph{Random Struc. and Alg.} \textbf{19},  (2001), 54--74.


\bibitem{cf:Bollobas} {\sc B. Bollob\'as}. The isoperimetric number of random regular graphs. \emph{Europ. J. Combin.} \textbf 9,  (1988), no. 3, 241--244.

\bibitem{cf:BollobasChung} {\sc B. Bollob\'as and F.R.K. Chung}. The Diameter of a Cycle Plus a Random Matching
\emph{SIAM J. Discrete Math.} \textbf 1,  (1988), no. 3, 328--333.

\bibitem{Bollobas_book} {\sc B. Bollob\'as}. \emph{Random graphs}. Second edition. Cambridge Studies in Advanced Mathematics, 73. Cambridge University Press, Cambridge, 2001.

\bibitem{cf:Cox} {\sc J.T. Cox}. Coalescing random walk and voter model consensus time on the torus in $\mathbb{Z}^{d}$. \emph{Ann. Probab.} \textbf{17},  (1989), no. 4, 1333--1366.

\bibitem{cf:Durrett_steppingI} {\sc J.T. Cox and R. Durrett}. The stepping stone model. New formulas expose old myths. \emph{Ann. Appl. Probab}.\textbf{12}, (2002),  no. 4, 1348--1377..

\bibitem{cf:CoxGriffeath} {\sc J.T. Cox and D. Griffeath}. Occupation time limit theorems for the 
voter model. Diffusive clustering in the two dimensional voter model. 
\emph{Ann. Probab.} \textbf{11},  (1983), no. 4, 876--893. .


\bibitem{cf:Durrett} {\sc R. Durrett}. \emph{Random Graph Dynamics}. Cambridge Series in Statistical and Probabilistic Mathematics. Cambridge University Press, Cambridge, 2007.

\bibitem{cf:DurrettJung} {\sc R. Durrett and P. Jung}. Two phase transitions for the contact process on small worlds. \emph{Stochastic Process. Appl.} \textbf{117}, (2007), no. 12, 1910--1927.
 
\bibitem{cf:Flatto} {\sc L. Flatto, A.M. Odlizko and D.B. Wales}. Random shuffles and group representations. \emph{Ann. Probab} \textbf 3, (1985), no. 1, 154--178. 

\bibitem{cf:NewmannWatts} {\sc M.E.J. Newmann and D.J. Watts}. Renormalization group analysis ans small-world network model. \emph{Physics Letters A.} \textbf{263}, (1999), 341--346.

\bibitem{cf:SinJer89} {\sc A. Sinclair and M. Jerrum}. Approximate counting, uniform generation and rapidly mixing Markov chains. \emph{Inform. and Comput.}  \textbf{82}, (1989), 93--133.

\bibitem{cf:Tavare} {\sc S. Tavar\'e}. Line-of-descent and genealogical processes and their applications in population genetics models. \emph{Theoret. Population. Biol.} \textbf{26}, (1984), no. 2, 119--164.

\bibitem{cf:WattsStrogatz} {\sc D.J. Watts and H. Strogatz}. Collective dynamics of small world networks. \emph{Nature}. \textbf{393}, 440--442.

\bibitem{cf:Woess} {\sc W. Woess}. \emph{Random Walks on Infinite Graphs and Groups}. Cambridge Tracts in Mathematics, 138. Cambridge University Press, Cambridge, 2000
 
\end{thebibliography}
\end{document}